\documentclass[a4paper,11pt]{smfart}

\usepackage[T1]{fontenc}        
\usepackage[latin1]{inputenc}   
\usepackage[francais, english]{babel}
\usepackage{lmodern}
\usepackage{amssymb,amsfonts,amsthm,amsmath,latexsym}
\usepackage{enumerate}

\usepackage[margin=1.1in]{geometry}

\theoremstyle{plain}
\newtheorem{theoreme}{Théorème}[section]

\newtheorem{lemme}{Lemme}[section]
\newtheorem{prop}{Proposition}[section]

\newtheorem{lemmeref}{Lemme}

\theoremstyle{definition}
\newtheorem*{ack}{Remerciements}

\theoremstyle{remark}
\newtheorem*{remarque}{Remarque}

\newcommand{\e}{{\rm e}}
\newcommand{\dd}{{\rm d}}

\newcommand{\cC}{{\mathcal C}}
\newcommand{\cPhi}{{\check{\Phi}_0}}

\newcommand{\cD}{{\mathcal D}}
\newcommand{\cL}{{\mathcal L}}

\newcommand{\cY}{{\mathcal Y}}

\newcommand{\ee}{{\varepsilon}}

\newcommand{\bfN}{{\mathbf N}}
\newcommand{\bfZ}{{\mathbf Z}}

\newcommand{\bfR}{{\mathbf R}}
\newcommand{\bfC}{{\mathbf C}}
\newcommand{\bfUn}{{\mathbf 1}}

\newcommand{\vth}{{\vartheta}}
\newcommand{\vphi}{{\varphi}}

\newcommand{\card}{{\rm card}}
\newcommand{\nv}{{\rm nv}}
\newcommand{\Vt}{{\widetilde V}}

\newcommand{\floor}[1]{{\left\lfloor {#1} \right\rfloor}}

\newcommand{\DSZ}{\nu}

\renewcommand{\mod}[1]{({\rm mod\ }#1)}

\renewcommand\Re{\operatorname{\mathfrak{Re}}}
\renewcommand\Im{\operatorname{\mathfrak{Im}}}

\numberwithin{equation}{section}

\title{Sommes friables d'exponentielles et~applications}
\author{Sary Drappeau}
\date{\today}

\address{Université Paris Diderot - Paris 7 \\ Institut de Mathématiques de Jussieu--Paris Rive Gauche \\ UMR 7586 \\
Bâtiment Chevaleret \\ Bureau~7C08 \\ 75205 Paris Cedex 13}
\email{drappeau@math.jussieu.fr}

\begin{altabstract}
An integer is said to be~$y$-friable if its greatest prime factor is less than~$y$.
In this paper, we obtain estimates for exponential sums over~$y$-friable numbers up to~$x$ which are non-trivial
when~$y \geq \exp\{c \sqrt{\log x} \log \log x\}$. As a consequence, we obtain an asymptotic formula for the
number of~$y$-friable solutions to the equation~$a+b=c$ which is valid unconditionnally under the same assumption.
We use a contour integration argument based on the saddle point method, as developped in the context of friable numbers by Hildebrand \& Tenenbaum,
and used by Lagarias, Soundararajan and Harper to study exponential and character sums over friable numbers.
\end{altabstract}

\begin{document}

\maketitle

\setcounter{tocdepth}{2}

\tableofcontents

\section{Introduction}

Soit~$P(n)$ le plus grand facteur premier d'un entier~$n>1$, avec la convention~$P(1)=1$.
Un entier~$n\geq 1$ est dit $y$-friable si~$P(n) \leq y$.
On note
\[ S(x, y) = \{ n \leq x \mid P(n) \leq y \} \]
l'ensemble des entiers~$y$-friables inférieurs ou égaux à~$x$. Le cardinal~$\Psi(x, y)$ de cet ensemble a fait l'objet d'abondantes études,
les techniques variant suivant le domaine en~$x$ et~$y$ auquel on s'intéresse (\emph{cf.} les articles de survol
de Hildebrand \& Tenenbaum~\cite{TeneHild1993} et Granville~\cite{granville2008smooth} qui exposent de façon exhaustive les travaux antérieurs).
Le problème qui nous intéresse est l'étude des sommes d'exponentielles tronquées sur les friables
\[ E(x, y ; \vth) := \sum_{n\in S(x,y)} \e(n\vth) \]
où l'on note~$\e(t) := \e^{2i\pi t}$.
Le comportement de $E(x, y ; \vth)$ diffère selon le degré de proximité de $\vth$ avec un rationnel de petit dénominateur.
Pour tout entier~$Q \geq 3$ et tout réel~$\vth$, il existe au moins un rationnel~$a/q$ avec
\[ (a, q)=1, \quad q\leq Q, \quad \left|\vth - \frac{a}{q} \right| \leq \frac{1}{qQ} .\]
On note $q(\vth, Q)$ le plus petit des dénominateurs $q$ pour lesquels une fraction $a/q$ vérifie cela ; dans ce cas~$a = a(\vth, Q)$ est unique.
Lorsque~$\vth$ est irrationnel, on a
\[ \lim_{Q \to \infty} q(\vth, Q) = \infty .\]
Une question intéressante est de déterminer dans quelle mesure la relation
\begin{equation}\label{E_negl}
E(x, y ; \vth) = o(\Psi(x, y))
\end{equation}
est valable lorsque~$x$ et~$y$ tendent vers l'infini, avec~$\vth$ irrationnel.
Fouvry et Tenenbaum~\cite[théorème~10]{fouvrytene1991entiers} montrent que la relation~\eqref{E_negl} a lieu pour tout $\delta>0$
et~$\vth$ irrationnel fixés lorsque $x$ et $y$ tendent vers l'infini en vérifiant
\[ x^{\delta (\log \log \log x) / \log \log x} \leq y \leq x .\]
La Bretèche~\cite[corollaires 4 et 5]{rdlb1998expos} montre la validité de~\eqref{E_negl}
pour tout~$\vth$ irrationnel fixé lorsque~$x$ et~$y$ tendent vers l'infini en vérifiant
\[ \exp\{c(\log x \log \log x)^{2/3}\} \leq y \leq x \]
pour une certaine constante~$c>0$.
L'argument présenté ici permet d'étendre encore le domaine de validité de~\eqref{E_negl}. On définit le domaine
\begin{equation}\label{domaine_xy}\tag{$\cD_{c}$}
\exp\{ c (\log x)^{1/2} \log \log x \} \leq y \leq x
\end{equation}

\begin{theoreme}\label{thm_E_negl}
Il existe une constante~$c>0$ telle que la relation~\eqref{E_negl} soit valable pour tout~$\vth$ irrationnel fixé
lorsque~$x$ et~$y$ tendent vers l'infini en restant dans le domaine~$\cD_c$.
\end{theoreme}

Le Théorème~\ref{thm_E_negl} découle d'une estimation asymptotique plus précise de la quantité~$E(x, y ; \vth)$ ; afin de l'énoncer, on définit
\[ u := (\log x)/\log y \]
\[ H(u) := \exp\left\{\frac{u}{\log(u+1)^2}\right\} \qquad (u \geq 1) \]
\[ \zeta(s, y) := \sum_{P(n) \leq y} n^{-s} = \prod_{p\leq y}\left(1 - p^{-s}\right)^{-1} \qquad (\sigma>0) .\]
En remarquant que~${\bfUn}_{[1, x]}(n) \leq (x/n)^\sigma$ pour tout~$\sigma>0$, on obtient la majoration de
Rankin~\cite{rankin1938difference},
\[ \Psi(x, y) \leq x^\sigma \zeta(\sigma, y) \qquad (2\leq y \leq x, \sigma>0) .\]
Le membre de droite est minimal lorsque $\sigma = \alpha = \alpha(x, y)$, la solution à
\[ \sum_{p\leq y} \frac{ \log p}{p^{\alpha}-1} = \log x .\]
Pour $x$ et $y$ suffisamment grands, on a $0 < \alpha < 1$, et plus précisément lorsque~$2\leq y \leq x$,
\begin{equation}\label{estim_alpha}
\alpha(x, y) = \frac{\log(1+y/\log x)}{\log y} \left\{ 1 + O\left(\frac{\log \log(1+y)}{ \log y} \right) \right\}
,\end{equation}
voir par exemple~\cite[theorem 2]{TeneHild86}. On a par ailleurs (\emph{cf.}~\cite[lemma 2]{TeneHild86}),
\begin{equation}\label{majo_alpha_1}
\alpha = 1 + O(\log(u+1)/\log y)
.\end{equation}
La majoration de Rankin fournit en fait une majoration de bonne qualité de~$\Psi(x, y)$ :
elle n'est qu'à un facteur~$O(\log x)$ de l'ordre de grandeur exact, obtenu par Hildebrand et Tenenbaum
par la méthode du col (\emph{cf.}~\cite[theorems 1,~2]{TeneHild86}, formule~\eqref{estim_psi_alpha} \emph{infra}).
Dans ce contexte, le réel~$\alpha$ joue le rôle du point-selle.

De Bruijn~\cite{deBruijn1951} puis Saias~\cite{Saias1989} ont obtenu une estimation de~$\Psi(x, y)$ très précise pour les grandes valeurs de~$y$.
On définit
\[ \Lambda(x, y) := \begin{cases} x\int_{-\infty}^\infty \rho(u-v) \dd(\floor{y^v}/y^v) & \mbox{si } x \in \bfR \setminus \bfN \\
\Lambda(x+0, y) & \mbox{si } x\in \bfN \end{cases} \]
où~$u\mapsto \rho(u)$ est la fonction de Dickman, l'unique solution continue sur~$]0, \infty[$ de l'équation différentielle
aux différences~$u\rho'(u)+\rho(u-1)=0\ (u>1)$ satisfaisant~$\rho(u) = 1 \ (u\in[0,1])$. Pour tout~$\ee>0$, dans le domaine~$(H_\ee)$
défini par
\begin{equation}\label{def_He}\tag{$H_\ee$}
3 \leq \exp\{(\log\log x)^{5/3+\ee}\} \leq y \leq x
\end{equation}
on a
\begin{equation}\label{estim_Psi_saias}
\Psi(x, y) = \Lambda(x, y)\{ 1 + O_\ee(\cY_\ee^{-1}) \}
\end{equation}
où l'on a posé pour tout $\ee>0$,
\begin{equation}\label{def_Ye}
\cY_\ee := \exp\{(\log y)^{3/5-\ee}\}
\end{equation}
On pose également, de même que dans~\cite{AGRB2010},
\[ \lambda(t, y) := \frac{\Lambda(t,y)}{t}  + \frac{1}{\log y}\int_{-\infty}^\infty \rho'((\log t)/\log y - v) \dd(\floor{y^v}/y^v) .\]
On a l'égalité entre mesures
\begin{equation}\label{Lambda_lambda}
\dd \Lambda(t, y) = \lambda(t, y)\dd t - t \dd(\{t\}/t)
.\end{equation}
Par ailleurs, la quantité~$\lambda(t, y) - y\{t/y\}/(t \log y)$ est dérivable par rapport à~$t$ pour tout~$t\geq y$.
On note~$\lambda'(t, y)$ cette dérivée.

On reprend les notations de La Bretèche et Granville~\cite{AGRB2010} pour la région sans zéro des fonctions $L$ de Dirichlet,
et du zéro exceptionnel.
Il existe une constante $b>0$ telle que pour tout $Q \geq 2$ et~$T\geq 2$,
la fonction $s \mapsto L(s, \chi)$ n'admette pas de zéro dans la région
\begin{equation}\label{RSZ}
\left\{ s = \sigma + i \tau \in \bfC \mid \sigma \geq 1 - \frac{b}{\log(Q T)} \mbox{ et } |\tau| \leq T \right\}
\end{equation}
pour tous les caractères $\chi$ de modules $q$ avec $1 \leq q \leq Q$ sauf éventuellement pour des caractères
tous associés à un même caractère primitif $\chi_1$ de module noté $q_1$.
Si ce caractère existe, il est quadratique et le zéro exceptionnel, noté~$\beta$, est unique, simple et réel ;
si pour une même valeur de~$Q$ et deux valeurs distinctes de~$T$, un tel caractère existe, alors il s'agit du même
et on dira que ce caractère est~$Q$-exceptionnel, tout en notant que pour des valeurs de~$T$ suffisamment grandes en fonction de~$Q$,
un tel caractère n'existe pas.
Le \og caractère de module 1 \fg\ désigne ici le caractère trivial, et la fonction $L$ associée est la fonction~$s \mapsto \zeta(s)$.
On note~$\chi_r$ le caractère de module~$q_1 r$ associé à~$\chi_1$, et on pose
$q \mapsto \DSZ(q)$ la fonction indicatrice des entiers multiples de~$q_1$ si~$\beta$ existe, et la fonction nulle sinon.
On garde dans toute la suite la notation~$s = \sigma + i\tau$.

On désigne par~$\cPhi(\lambda, s)$ la fonction définie pour~$\sigma>0$ par
\begin{equation}\label{def_cPhi}
\cPhi(\lambda, s) := \int_0^1 \e(\lambda t) t^{s-1} \dd t
.\end{equation}
En développant en série entière le terme~$\e(\lambda t)$ on obtient~$\cPhi(\lambda, s) = \sum_{n\geq 0} (2i\pi\lambda)^n / ((n+s) n!)$,
et cela permet de prolonger~$s\mapsto\cPhi(\lambda,s)$ en une fonction méromorphe sur~$\bfC$, qui possède un pôle simple en~$s=0$ de résidu~$1$,
et lorsque~$\lambda\neq 0$, un pôle simple en~$s=-n$ pour tout entier~$n\geq 1$, de résidu~$(2i\pi \lambda)^n/n!$.

Enfin on note respectivement~$\omega(n)$ et~$\tau(n)$ le nombre de facteurs premiers et le nombre de diviseurs d'un entier~$n \geq 1$,
et on pose
\[ \cL := \exp\sqrt{\log x}, \]
\begin{equation}\label{defs_t1_t2}
T_1 = T_1(x, y) := \min\{y, \cL\}, \qquad T_2=T_2(x, y) := \min\{y^{1/\log \log \log x}, \cL\}
.\end{equation}

\begin{theoreme}\label{thm_E_pcp}
Il existe des constantes $c_1$ et $c_2$ positives et une fonction $W(x, y ; q, \eta)$
telles que pour tout $(x, y)$ dans le domaine
\begin{equation}\label{domaine_xy*}
(\log x)^{c_1} \leq y \leq \exp\{(\log x) / (\log \log x)^4 \},
\end{equation}
pour tout~$\vth \in \bfR$ avec~$\vth=a/q+\eta$ où~$(a, q)=1$, $q \leq T_2^{c_2}$ et~$|\eta| \leq T_2^{c_2}/x$ on ait
\begin{equation}\label{estim_E_pcp}
\begin{aligned}
E(x, y ; \vth) = &\ \frac{\alpha q^{1-\alpha}}{\vphi(q)} \prod_{p | q} (1-p^{\alpha-1}) \cPhi(\eta x, \alpha) \Psi(x, y)
+ \DSZ(q) \chi_1(a) W(x, y ; q, \eta) \\
& + O\left( \frac{2^{\omega(q)} q^{1-\alpha} \Psi(x, y)}{\vphi(q)(1+|\eta|x)^\alpha }\frac{(\log q)^2 (\log (2+|\eta|x))^3}{u}
 + \frac{\Psi(x, y)}{T_2^{c_2}} \right)
\end{aligned}
\end{equation}
où~$\chi_1$ est l'éventuel caractère~$T_2^{c_2}$-exceptionnel, et avec, dans le cas~$\DSZ(q)\neq 0$,
\begin{equation}\label{estim_W}
W(x, y ; q, \eta) \ll \frac{2^{\omega(q/q_1)}\sqrt{q_1}(q/q_1)^{1-\alpha}\Psi(x,y)}{\vphi(q)(1+|\eta|x)^\alpha x^{1-\beta}H(u)^{c_2}} 
+ \frac{2^{\omega(q/q_1)}\sqrt{q_1}(q/q_1)^{1-\alpha}\Psi(x, y)}{\vphi(q)T_2^{c_2}}
.\end{equation}

Si de plus on a~$(x, y) \in (H_\ee)$, $q \leq \cY_{\ee}$ et~$|\eta| \leq \cY_{\ee}/(qx)$ pour un certain~$\ee>0$, alors
\begin{equation}\label{estim_E_BG}
\begin{aligned}
E(x, y ; \vth) = & \ \Vt(x, y ; q, \eta) + \DSZ(q) \chi_1(a) W(x, y ; q, \eta) \\
& + O_\ee\left( \frac{2^{\omega(q)} \Psi(x, y)}{\vphi(q) \cY_{\ee}} + \frac{\Psi(x, y)}{T_2^{c_2}} \right)
\end{aligned}
\end{equation}
où l'on a posé
\begin{equation}\label{def_Vt}
\Vt(x, y ; q, \eta) :=
\sum_{k|q} \frac{\mu(q/k)k}{ \vphi(q) } \sum_{\substack{n\leq x\\k|n}} \e(n\eta) \lambda\left(\frac{n}{k}, y\right)
.\end{equation}
\end{theoreme}

\begin{remarque}
Les conditions sur~$q$ et~$\eta$ dans l'estimation~\eqref{estim_E_pcp} sont moins restrictives que celles de~\eqref{estim_E_BG},
mais son terme d'erreur est moins bon.
\end{remarque}

Les estimations du Théorème~\ref{thm_E_pcp} ne sont valables que lorsque~$\vth$ est proche d'un rationnel à petit dénominateur,
ce qui correspond aux arcs majeurs dans la terminologie de la méthode du cercle. Les valeurs complémentaires de~$\vth$ sont traitées
à l'aide du résultat suivant, déduit de~\cite[corollaire~3]{rdlb1998expos}.
\begin{lemmeref}[\cite{rdlb1998expos}, corollaire~3]\label{majo_E_arcmin}
Lorsque les réels~$\vth, x, R$ vérifient~$x, R\geq 2$ et $q(\vth, \lceil x/R \rceil) \geq R$, on a
\[ E(x, y ; \vth) \ll x (\log x)^4 \{ 1/R^{1/4} + 1/\cL \} .\]
\end{lemmeref}

\begin{proof}[Démonstration du Théorème~\ref{thm_E_negl}]
Soient~$c_1$ et~$c_2$ les constantes données par le~Théorème~\ref{thm_E_pcp} et supposons~$y^{1/\log\log\log x}\geq \cL$,
de sorte que~$T_1 = T_2 = \cL$. On pose~$q = q(\vth, \lceil x/\cL^{c_2} \rceil)$ et~$\theta = a/q + \eta$
avec~$\eta \leq 1/(q\lceil x/\cL^{c_2} \rceil)$ ; on a~$|\eta|x \leq \cL^{c_2}$. Lorsque~$y \geq \exp\{(\log x)/(\log \log x)^4 \}$,
les résultats de La Bretèche~\cite[théorème~2]{rdlb1998expos} s'appliquent,
on suppose donc sans perte de généralité que~$y \leq \exp\{(\log x)/(\log \log x)^4 \}$. Lorsque~$q> \cL^{c_2}$, d'après le Lemme~\ref{majo_E_arcmin}
on a~$E(x, y ; \vth) \ll x/\cL^{c_3}$ pour une constante~$c_3>0$. Pour~$\exp\{c\sqrt{\log x}\log \log x\} \leq y$ avec~$c$ suffisamment grande,
cela est~$o(\Psi(x, y))$. Enfin, lorsque~$q \leq \cL^{c_2}$, l'estimation~\eqref{estim_E_pcp} est valable
et tous les termes du membre de droite sont~$o(\Psi(x, y))$ quand~$q \to \infty$ et~$x \to \infty$,
en remarquant que~$\cPhi(\eta x, \alpha) \ll 1$.
\end{proof}

La démonstration que l'on propose du Théorème~\ref{thm_E_pcp} utilise une majoration du type~$H(u)^{-\delta} (\log x) \ll_\delta 1$
pour tout~$\delta>0$ fixé, qui n'est pas valable lorsque~$y$ est trop proche de~$x$. Ceci explique la borne supérieure en~$y$
du domaine~\eqref{domaine_xy*}.

Le domaine en~$x$ et~$y$ dans lequel on peut majorer non trivialement~$E(x, y ; \vth)$ pour~$\vth$ irrationnel a une influence directe
sur le domaine de validité de certains résultats qui sont liés aux sommes d'exponentielles.
On en cite deux ; le premier est une généralisation d'un théorème de Daboussi~\cite{daboussi1975fonctions}.
\begin{theoreme}\label{thm_daboussi}
Il existe une constante~$c>0$ telle que pour toute fonction~$Y: [2,\infty[ \to \bfR$ croissante avec $(Y(x), x) \in \cD_{c}$,
toute fonction~$f:\bfN \to \bfC$ multiplicative satisfaisant pour tous~$x$ et~$y$
avec~$Y(x) \leq y \leq x$,
\[ \sum_{n\in S(x,y)} |f(n)|^2 \leq K_f \Psi(x, y) \]
pour un certain réel~$K_f>0$ dépendant au plus de~$f$,
et tout~$\vth$ irrationnel, lorsque~$x$ et~$y$ tendent vers l'infini avec~$Y(x) \leq y \leq x$, on ait
\[ \sum_{n \in S(x, y)} f(n)\e(n\vth) = o_{\vth}(K_f^{1/2}\Psi(x, y)) .\]
\end{theoreme}
Cela est une extension de~\cite[théorème~1.5]{tenerdlb2005turan}. Suivant Dupain, Hall et Tenenbaum~\cite{dupain1982equirepartition},
on peut se poser la question de savoir pour quelle classe de fonctions multiplicatives~$f$ et
quelles suites d'ensembles finis d'entiers~$(E_N)_{N \geq 1}$ la relation
\begin{equation*}
\sum_{n \in E_N} f(n) \e(n\vth) = o\Big(\sum_{n \in E_N} |f(n)|\Big)
\end{equation*}
est valable pour tout~$\vth$ irrationnel fixé lorsque~$N\to \infty$. Le Théorème~\ref{thm_daboussi} aborde le cas particulier~$E_N = S(N, y_N)$
avec~$Y(N)\leq y_N \leq N$.

La deuxième application que l'on considère concerne le problème du comptage des solutions friables à l'équation~$a+b=c$. Posons
\begin{equation}\label{def_Nxy}
N(x, y) := \card \{(a, b, c) \in S(x, y)^3 \mid a+b=c\}
.\end{equation}
Lagarias et Soundararajan étudient cette quantité dans~\cite{SoundLaga2011}.
Leur travail, précisé par l'auteur~\cite{D2012}, implique en particulier qu'en supposant l'hypothèse de Riemann
généralisée aux fonctions $L$ de Dirichlet, on a
\begin{equation}\label{N_sim}
N(x, y) \sim \frac{\Psi(x, y)^3}{2 x}
\end{equation}
lorsque $(\log \log x)/\log y \to 0$. Dans~\cite{AGRB2010}, La Bretèche et Granville obtiennent inconditionnellement,
à partir des estimations de~$E(x, y ; \vth)$ démontrées dans~\cite{rdlb1998expos}, que la relation~\eqref{N_sim} est valable,
pour tout~$\ee>0$ fixé, lorsque~$x$ et~$y$ tendent vers l'infini avec~$\exp\{(\log x)^{2/3+\ee} \} \leq y \leq x$.
Les estimations de~$E(x, y ; \vth)$ présentées ici permettent d'étendre le domaine de validité de cette estimation.
\begin{theoreme}\label{thm_abc}
Il existe~$c>0$ tel que lorsque~$(x, y) \in \cD_{c}$, on ait
\begin{equation}\label{estim_N}
 N(x, y) = \frac{\Psi(x, y)^3}{2x} \left\{ 1 + O\left(\frac{\log(u+1)}{\log y}\right) \right\}
.\end{equation}
\end{theoreme}
\begin{remarque}
Le terme d'erreur dans l'estimation~\eqref{estim_N} est attendu comme optimal. On peut, à la façon de Saias~\cite{Saias1989},
obtenir un développement du membre de gauche selon les puissances de~$(\log y)^{-1}$.

\medskip

Dans~\cite{AGRB2010}, les auteurs étudient la densité sur les friables d'une suite générale satisfaisant
des hypothèses de crible. Cette application n'est pas développée ici mais le Théorème~\ref{thm_E_pcp} permet d'étendre leur
résultat à tout~$(x, y) \in \cD_{c}$ pour un certain~$c>0$.
\end{remarque}

\begin{ack}
L'auteur adresse ses vifs remerciements son directeur de thèse Régis de la Bretèche pour sa grande patience et ses nombreux conseils,
et à Adam Harper pour des remarques qui ont aidé à améliorer ce manuscrit.
\end{ack}

\section{Estimation de~$E(x, y ; \vth)$}

\subsection{Méthode du col}

Soit à étudier la fonction sommatoire sur les entiers friables d'une suite de nombres complexes de modules~$\leq 1$
\[ A(x, y) = \sum_{n \in S(x, y)} a_n \]
lorsque~$x\not\in \bfN$, prolongée par~$A(x, y) := A(x-0, y) + a_x/2$ lorsque~$x$ est un entier~$y$-friable.
La série de Dirichlet associée
\begin{equation}\label{def_serie_dir}
F(s, y) := \sum_{P(n)\leq y} a_n n^{-s}
\end{equation}
converge absolument lorsque~$\sigma>0$. En appliquant la formule de Perron, on écrit
\[ A(x, y) = \frac{1}{2i\pi} \int_{\kappa-i\infty}^{\kappa+i\infty}{F(s, y)x^s\frac{\dd s}{s} } \]
où~$\kappa>0$ est fixé.
La méthode du col consiste à modifier le chemin d'intégration pour faire en sorte que la contribution principale à l'intégrale
entière vienne d'une petite partie du chemin d'intégration, suffisamment petite pour pouvoir l'estimer par une formule de Taylor.
Dans le cas~$a_n=1$, où il s'agit essentiellement d'estimer~$\Psi(x, y)$, on intègre sur la droite~$\sigma=\alpha$.
Le point~$\alpha$ est le minimum de la fonction~$\sigma\mapsto x^\sigma\zeta(\sigma, y)$ et sa dérivée seconde en ce point est non nulle,
le point~$\tau=0$ est donc un maximum local de la fonction~$\tau\mapsto |x^{\alpha+i\tau}\zeta(\alpha+i\tau, y)|$. On définit
\[ \sigma_2 = \sigma_2(x, y) := \sum_{p\leq y} \frac{p^\alpha (\log p)^2}{(p^\alpha-1)^2} \]
qui est la valeur en~$\alpha$ de la dérivée seconde de la fonction $s\mapsto \log\zeta(s, y)$.
Lorsque~$2\leq y\leq x$, on a d'après~\cite[theorem 2]{TeneHild86},
\begin{equation}\label{estim_sigma2}
 \sigma_2(x, y) = \log x \log y \left(1+\frac{\log x}{y}\right) \left\{ 1 + O\left(\frac{1}{\log(1+u)}+\frac{1}{\log y}\right) \right\}
.\end{equation}
Le résultat principal de~\cite{TeneHild86} est l'estimation, uniforme pour~$2 \leq y \leq x$,
\begin{equation}\label{estim_psi_alpha}
\Psi(x, y) = \frac{x^\alpha\zeta(\alpha, y)}{\alpha\sqrt{2\pi\sigma_2}} \left\{1+O\left(\frac{1}{u}+\frac{\log y}{y}\right)\right\}
.\end{equation}
Par rapport aux précédents résultats sur~$\Psi(x, y)$, cette estimation a l'avantage,
au prix d'un terme principal moins explicite, d'être valide sans aucune contrainte sur~$x$ et~$y$.
L'estimation~\eqref{estim_sigma2} implique en particulier que pour tout~$(x,y)$ avec~$2 \leq y \leq x$,
on a
\begin{equation}\label{x_zeta_psi}
\zeta(\alpha, y) x^\alpha \ll (\log x) \Psi(x, y)
.\end{equation}

Un autre intérêt de la méthode du col est qu'elle permet une étude uniforme du rapport~$\Psi(x/d, y)/\Psi(x, y)$,
ce qui est utile dans beaucoup d'applications. Cette question
ainsi que d'autres problèmes associés sont étudiés en détail dans~\cite{TeneRDLB2005Stat}.
\begin{lemmeref}[\cite{TeneRDLB2005Stat}, théorème 2.4]\label{estim_psi_local}
Il existe deux constantes positives~$b_1$ et~$b_2$  et une fonction~$b = b(x, y; d)$ satisfaisant~$b_1 \leq b \leq b_2$
telles que pour $\log x \leq y \leq x$ et~$1 \leq d \leq x$ on ait uniformément
\[
\Psi\left( \frac{x}{d}, y\right) = \left\{1 + O\left(\frac{t}{u}\right) \right\} \left(1 - \frac{t^2}{u^2}\right)^{b u} \frac{\Psi(x, y)}{d^\alpha}
\]
où l'on a posé $t = (\log d)/\log y$.
\end{lemmeref}
Cela implique sous les mêmes hypothèses la majoration
\begin{equation}\label{estim_psi_local_crude}
\Psi(x/d, y) \ll  \Psi(x, y)/d^\alpha,
\end{equation}
celle-ci étant valable pour tout~$d\geq 1$.

\subsection{Somme sur les caractères, formule de Perron}

Pour tout caractère de Dirichlet $\chi$ de module $q$, on définit la somme de Gauss $\tau(\chi) := \sum_{b \mod{q}}{\chi(b) \e(b/q)}$.
On a pour tous~$x$ et~$y$ avec~$x \geq y \geq 2$, $\vth \in \bfR$ et $(a, q) \in \bfZ \times \bfN$ avec $\vth = a/q + \eta$,
\begin{equation}\label{E_sum_chi}
E(x, y ;\vth) = \sum_{\substack{d | q \\ P(d) \leq y}}{\frac{1}{\vphi(q/d)} \sum_{\chi \mod{q/d}}{\chi(a) \tau(\overline{\chi}) \sum_{m \in S(x/d, y)}{\e(m d \eta) \chi(m)} } }
.\end{equation}
Une façon d'étudier~$E(x, y ; \vth)$ est donc d'obtenir des estimations uniformes de la somme
\begin{equation}\label{def_Psi0}
 \Psi_0(z, y ; \chi, \gamma) := \sum_{n \in S(z, y)}{\e(n \gamma) \chi(n)}
.\end{equation}
On rappelle que~$\cPhi(\lambda, s)$ et~$F(s, y)$ sont définis respectivement en~\eqref{def_cPhi} et~\eqref{def_serie_dir}.

\begin{lemme}\label{perron_phi}
Soit $(a_n)_{n\geq 1}$ une suite de nombres complexes telle que l'abscisse de convergence absolue de la série~$\sum_{P(n)\leq y} a_n n^{-s}$
soit strictement inférieure à $1/2$.
Lorsque~$x, y \geq 2$, $\eta \in \bfR$, $T\geq 2$, $\kappa\in[1/2, 1]$, $c \in ]0,1/2]$ et~$M\geq 0$,
et lorsque les inégalités suivantes sont satisfaites :
\begin{align*}
\sum_{P(n) \leq y} |a_n| n^{-\kappa} \leq M \zeta(\kappa, y) \qquad \mbox{et} \qquad
\sum_{\substack{P(n) \leq y \\ |n - x| < x/\sqrt{T}}} |a_n|  \leq M \Psi(x, y) / T^c,
\end{align*}
on a uniformément
\begin{equation}\label{perron_phi_estim}
\begin{aligned}
\sum_{n \in S(x, y)}{a_n \e(n\eta)} = &\ \frac{1}{2i\pi} \int_{\kappa-iT}^{\kappa+iT} F(s, y) x^s \cPhi(\eta x, s) \dd s \\
&\ + O\left( M (\log T)(1+|\eta x|) \frac{x^\kappa \zeta(\kappa, y)}{T^c} \right).
\end{aligned}
\end{equation}
\end{lemme}
\begin{remarque}
En particulier, si l'on suppose que la suite~$(a_n)$ est bornée, un théorème de Hildebrand sur le nombre des friables
dans les petits intervalles~\cite[theorem~4]{hildebrand1985integers} ainsi que la majoration~\eqref{estim_psi_local_crude}
assurent que les hypothèses sur~$(a_n)_{n\geq 1}$ sont satisfaites pour~$M$ absolu et~$c=\alpha(x, y)/2$.
Lorsque~$(\log x)^K \leq y \leq x$ pour un certain~$K>1$ fixé, on a~$\alpha(x, y) \gg_K 1$.
\end{remarque}

Le Lemme~\ref{perron_phi} découle du lemme suivant, qui est une généralisation
d'un lemme classique de Perron (\emph{cf.}~\cite[lemme II.2.2]{Tene2007}).

\begin{lemme}\label{perron_phi_lemme}
Pour tous réels $x, \kappa, T$ et~$\lambda$ avec $x\geq 0$, $\kappa \in [1/2, 1]$ et~$T \geq 2$, on a
\[ \left| {\bfUn}_{[1,\infty[}(x) \e(\lambda/x) - \frac{1}{2 i \pi}
\int_{\kappa - iT}^{\kappa + iT}{x^s \cPhi(\lambda, s) \dd s} \right| \ll \frac{(\log T) (1+|\lambda|) x^\kappa}{1+ T|\log x|} .\]
\end{lemme}

On énonce pour cela un lemme qui fournit des informations sur la taille de~$\cPhi$.

\begin{lemme}\label{majo_phi}
Pour tous $s \in \bfC$ et $\lambda \in \bfR$ avec $\sigma \geq 1/2$, on a
\[ \cPhi(\lambda, s) \ll \min\left\{ \frac{1}{\sigma}, \quad
\frac{|s|}{\sigma}\log(2+ |\lambda|)\big(|\lambda|^{-\sigma}+|\lambda|^{-1}\big), \quad
\frac{1+|\lambda|/\sigma}{|s|} \right\} .\]
\end{lemme}

\begin{proof}
On a trivialement~$\cPhi(\lambda, s) \ll 1/\sigma$. On a d'une part lorsque $|\lambda|\geq1$,
\begin{align*}
\int_0^1{\e(\lambda t) t^{s-1}\dd t} =&\ 
\left[ \frac{\e(\lambda t)-1}{2 i \pi \lambda} t^{s-1} \right]_0^1  - (s-1)\int_0^1{\frac{\e(\lambda t)-1}{2 i \pi \lambda} t^{s-2} \dd t}  \\
&\ \ll |\lambda|^{-1} + |s-1|\Big(\frac{|\lambda|^{-\sigma}}\sigma + \frac{|\lambda|^{-1}-|\lambda|^{-\sigma}}{\sigma-1}\Big) \\
&\ \ll \frac{|s|}{\sigma}\log(2+ |\lambda|)\big(|\lambda|^{-\sigma}+|\lambda|^{-1}\big)
\end{align*}
en séparant l'intégrale selon la position de~$t$ par rapport à~$1/|\lambda|$, et d'autre part, pour tout~$\lambda$,
\begin{align*}
\int_0^1{\e(\lambda t) t^{s-1}\dd t} =
\left[ \e(\lambda t) \frac{t^s}{s} \right]_0^1  - \frac{1}{s}\int_0^1{(2 i \pi \lambda) \e(\lambda t) t^{s} \dd t} 
\ll \frac{1+|\lambda|/\sigma}{|s|}
.\end{align*}
Le résultat suit en notant que $|s| \log(2+ |\lambda|) / |\lambda|^{\sigma} \gg 1$ pour~$|\lambda|<1$.
\end{proof}

\begin{proof}[Démonstration du lemme~\ref{perron_phi_lemme}]

Le cas~$\lambda=0$ étant démontré dans~\cite[lemme II.2.2]{Tene2007}, on suppose~$\lambda\neq 0$.
On rappelle que la fonction $s \mapsto \cPhi(\lambda, s)$ est prolongeable en une fonction méromorphe sur~$\bfC$
ayant pour tout~$n\geq 0$ un pôle simple en~$s=-n$, de résidu~$(2i\pi\lambda)^n/n!$.
On suppose~$x<1$. Pour tout réel~$k\geq 0$, en intégrant sur le rectangle de côtés
\[ \kappa+k \pm iT, \kappa \pm iT \]
on obtient grâce aux majorations du Lemme~\ref{majo_phi},
\begin{equation}\label{perronphi_x_petit}
\int_{\kappa-iT}^{\kappa+iT} x^s \cPhi(\lambda, s) \dd s \ll
\frac{(1+|\lambda|) x^\kappa (1-x^k)}{T |\log x|} + \frac{T(1+|\lambda|) x^{\kappa+k}}{k+\kappa} 
\ll \frac{(1+|\lambda|) x^\kappa}{T |\log x|}
\end{equation}
en faisant tendre~$k$ vers l'infini.

Pour $x\geq 1$, d'après la définition de~$\cPhi(\lambda, s)$, on a
\begin{align*}
I := \frac{1}{2i\pi} \int_{\kappa-iT}^{\kappa+iT} x^s \cPhi(\lambda, s) \dd s
&\ = \frac{1}{2\pi} \int_0^1 \left( \int_{-T}^T (t x)^{i\tau} \dd \tau \right) \e(\lambda t) \frac{(tx)^\kappa}{t} \dd t \\
&\ = \frac{1}{\pi} \int_{-\infty}^{T\log x} \frac{\sin w}{w}\e(\lambda \e^{w/T} / x) \e^{\kappa w/T} \dd w
\end{align*}
ayant posé $x t = e^{w/T}$. Une intégration par parties permet d'écrire~$I = I_1 + I_2 - I_3$ avec
\begin{align*}
I_1 := &\ \frac{1}{\pi} \int_{-\infty}^{T\log x} \frac{1-\cos w}{w^2} \e(\lambda \e^{w/T}/x) \e^{\kappa w/T} \dd w, \\
I_2 := &\ \frac{1-\cos(T\log x)}{\pi T \log x} \e(\lambda) x^{\kappa}, \\
I_3 := &\ \frac{1}{\pi} \int_{-\infty}^{T\log x} \frac{1-\cos w}{w} \left( \frac{2i\pi\lambda}{xT}\e^{w/T}
+ \frac{\kappa}{T}\right) \e(\lambda \e^{w/T}/x) \e^{\kappa w/T} \dd w
.\end{align*}
Des estimations élémentaires fournissent
\begin{align*}
I_1 = &\ \e(\lambda/x) - \frac{\e(\lambda/x)}{\pi} \int_{T\log x}^{\infty} \frac{1-\cos w}{w^2} \dd w \\
 &\ + \frac{1}{\pi} \int_{-\infty}^{T\log x} \frac{1-\cos w}{w^2} \left( \e(\lambda \e^{w/T}/x) \e^{\kappa w/T} - \e(\lambda/x) \right) \dd w \\
 = &\ \e(\lambda/x) + O\left( \frac{1}{1+T \log x} + \frac{x^\kappa}{T (1+(\log x)^2)}
+ \frac{\log T}{T} \left( 1 + \frac{|\lambda|}{x} \right) \right), \\
I_2 \ll &\ \frac{ T (\log x) x^\kappa}{1+(T\log x)^2}, \\
 I_3 \ll &\ \frac{\log T}{T}\left(1+\frac{|\lambda|}{x}\right)
+ \frac{(\log T) (1 + |\lambda|) x^\kappa}{T \log x}{\bfUn}_{[1, \infty[}(T \log x)
.\end{align*}

Ainsi, lorsque $x>1$, on a
\begin{equation}\label{perronphi_x_grand}
\left| \e(\lambda/x) - \int_{\kappa-iT}^{\kappa+iT} x^s \cPhi(\lambda, s) \dd s \right| = O\left( \frac{(\log T) (1+|\lambda|) x^\kappa}{T \log x} \right)
\end{equation}
et lorsque $x=1$, on a
\begin{equation}\label{perronphi_x_un}
\int_{\kappa-iT}^{\kappa+iT} \cPhi(\lambda, s) \dd s \ll 1 + \frac{(\log T) (1+|\lambda|)}{T} \ll 1+|\lambda|
.\end{equation}

Lorsque~$T |\log x| \geq 1$ l'estimation voulue découle de~\eqref{perronphi_x_petit} et \eqref{perronphi_x_grand}.
Si~$\e^{-1/T} < x < \e^{1/T}$, on a
\begin{align*}
\frac{1}{2i \pi} \int_{\kappa-iT}^{\kappa+iT} x^s \cPhi(\lambda, s) \dd s
 &\ = \frac{x^\kappa}{2 \pi} \int_{-T}^{T} \cPhi(\lambda, \kappa + i\tau) \dd \tau +
\frac{x^\kappa}{2 \pi} \int_{-T}^{T} (x^{i\tau} - 1) \cPhi(\lambda, \kappa + i\tau) \dd \tau \\
&\ \ll (1+|\lambda|) x^\kappa
\end{align*}
grâce à la majoration~\eqref{perronphi_x_un} et au Lemme~\ref{majo_phi}.
Cela implique
\[ \left| {\bfUn}_{[1,\infty[}(x) \e(\lambda/x) - \frac{1}{2 i \pi}
\int_{\kappa - iT}^{\kappa + iT}{x^s \cPhi(\lambda, s) \dd s} \right| \ll (1+|\lambda|) x^\kappa \]
ce qui fournit l'estimation voulue pour~$T|\log x| < 1$.
\end{proof}

\begin{remarque}
Un traitement plus fin de~$I_1$ permet d'obtenir dans le cas~$x=1$,
\[ \left| \frac{\e(\lambda)}{2} - \frac{1}{2i\pi} \int_{\kappa-iT}^{\kappa+iT} \cPhi(\lambda, s) \dd s \right| \ll \frac{(\log T)(1+|\lambda|)}{T} \]
mais cela ne sera pas utilisé ici.
\end{remarque}

\begin{proof}[Démonstration du lemme~\ref{perron_phi}]
Une application du Lemme~\ref{perron_phi_lemme} avec~$x$ remplacé par~$x/n$
et~$\lambda$ par~$\eta x$ permet d'écrire sous les hypothèses de l'énoncé,
\begin{align*}
\sum_{n \in S(x, y)}{a_n \e(n\eta)} = &\ \frac{1}{2i\pi} \int_{\kappa-iT}^{\kappa+iT} F(s, y) x^s \cPhi(\eta x, s) \dd s \\
&\ + O\left( (\log T)(1+|\eta x|) x^\kappa \sum_{P(n)\leq y}{ \frac{|a_n|n^{-\kappa}}{1+T|\log(x/n)|} } \right)
.\end{align*}
De même que dans~\cite[preuve du théorème~4]{fouvrytene1991entiers}, on sépare la somme dans le terme d'erreur selon la taille de $|\log(n/x)|$.
Les entiers~$n\in ]x-x/\sqrt T, x+x/\sqrt T[$ contribuent d'une quantité
\[ \ll (\log T)(1+|\eta x|)
\sum_{\substack{P(n)\leq y\\ x-x/\sqrt{T} < n \leq x+x/\sqrt{T}}} |a_n| \]
qui est de l'ordre du terme d'erreur annoncé grâce aux hypothèses sur~$(a_n)$ ainsi que la majoration~$\Psi(x, y) \leq x^\kappa \zeta(\kappa, y)$.
La contribution des entiers~$n \not\in ]x-x/\sqrt T, x+x/\sqrt T[$ est
\[ \ll (\log T)(1+|\eta x|)  \frac{x^\kappa}{\sqrt T} \sum_{P(n)\leq y} |a_n| n^{-\kappa} \]
qui est à nouveau de l'ordre du terme d'erreur annoncé.
\end{proof}

On montre enfin le résultat suivant, qui assure que dans le cadre des Propositions~\ref{caracts_pcp} et~\ref{caracts_sieg} \emph{infra},
les hypothèses du Lemme~\ref{perron_phi} sont vérifiées avec~$\kappa = \alpha$.
\begin{lemme}\label{hypoth_perron_a}
Soient~$q \geq 1$ un entier~$y$-friable, et $q_1$ un diviseur de~$q$. Soit~$\chi_1$ un caractère primitif modulo~$q_1$ et pour tout~$r \geq 1$,
$\chi_r$ le caractère modulo~$q_1r$ associé à~$\chi_1$. On note~$r_1 := q/q_1$ et on pose pour tout~$n\geq 1$,
\[ a_n := \frac{\tau(\chi_1) \mu\left(\frac{r_1}{(r_1, n)}\right) \chi_1\left(\frac{r_1}{(r_1, n)}\right)
\chi_{\frac{r_1}{(r_1,n)}}\left(\frac{n}{(r_1, n)}\right)}{\vphi(q_1) \vphi\left(\frac{r_1}{(r_1, n)}\right)} \]
Alors lorsque~$\kappa \in [1/2, 1]$, $2 \leq y \leq x$ et~$2\leq T\leq x$, on a uniformément
\begin{align*}
\sum_{P(n) \leq y} |a_n| n^{-\kappa} &\ \ll \frac{2^{\omega(q/q_1)} \sqrt{q_1}}{\vphi(q)} \left(\frac{q}{q_1}\right)^{1-\kappa} \zeta(\kappa, y), \\
\sum_{\substack{P(n) \leq y \\ |n-x| < x/\sqrt{T}}} |a_n| &\ \ll \frac{2^{\omega(q/q_1)} \sqrt{q_1}}{\vphi(q)} \left(\frac{q}{q_1}\right)^{1-\alpha}
\frac{\Psi(x,y)}{T^{\alpha/2}}
.\end{align*}
\end{lemme}

\begin{proof}
On a, en écrivant~$(n, r_1) = r_1/d$ et~$n = m r_1/d$,
\begin{align*}
\sum_{P(n) \leq y} |a_n| n^{-\kappa} &\ = \frac{\sqrt{q_1}}{\vphi(q_1)} \left(\frac{q}{q_1}\right)^{-\kappa}
\sum_{\substack{d|r_1 \\ (d, q_1)=1}} \frac{\mu^2(d) d^\kappa}{\vphi(d)} \sum_{\substack{P(m) \leq y \\ (m, q_1 d) = 1}} m^{-\kappa} \\
&\ = \frac{\sqrt{q_1}}{\vphi(q_1)} \left(\frac{q}{q_1}\right)^{-\kappa} \zeta(\kappa, y) 
\prod_{p | q_1}(1-p^{-\kappa}) \prod_{\substack{p | q/q_1 \\ p \nmid q_1}} \left(1 + \frac{p^\kappa-1}{p-1}\right) \\
&\ \leq \frac{\sqrt{q_1}}{\vphi(q)} \left(\frac{q}{q_1}\right)^{1-\kappa} \zeta(\kappa, y) 2^{\omega(q/q_1)}
.\end{align*}
Par ailleurs, avec les mêmes notations, on a
\begin{align*}
\sum_{\substack{P(n) \leq y \\ |n-x|\leq x/\sqrt{T}}} |a_n| &\ = \frac{\sqrt{q_1}}{\vphi(q_1)}
\sum_{\substack{d|r_1 \\ (d, q_1)=1}}  \frac{\mu^2(d)}{\vphi(d)}
\sum_{\substack{P(m) \leq y \\ \left|m r_1/d - x \right| \leq x/\sqrt{T} \\ (m, q_1 d) = 1}} 1 \\
&\ \leq \frac{\sqrt{q_1}}{\vphi(q_1)} \sum_{\substack{d|r_1 \\ (d, q_1)=1}}
\frac{\mu^2(d)}{\vphi(d)} \left(\Psi(x d(1+1/\sqrt{T}) / r_1, y) - \Psi(x d (1-1/\sqrt{T}) / r_1, y)\right) \\
&\ \leq 2 \frac{\sqrt{q_1}}{\vphi(q_1)} \sum_{\substack{d|r_1 \\ (d, q_1)=1}}  \frac{\mu^2(d)}{\vphi(d)} \Psi(x d / (r_1 \sqrt{T}), y)
\end{align*}
d'après~\cite[theorem~4]{hildebrand1985integers}. La majoration~\eqref{estim_psi_local_crude} a lieu
avec~$d$ remplacé par~$\sqrt{T} r_1 / d$ et ainsi
\begin{align*}
\sum_{\substack{P(n) \leq y \\ |\frac{n}{x}-1|\leq1/\sqrt{T}}} |a_n|
&\ \ll \frac{\sqrt{q_1} \Psi(x, y)}{\vphi(q_1) T^{\alpha/2}} \left(\frac{q}{q_1}\right)^{-\alpha}
\sum_{\substack{d|r_1 \\ (d, q_1)=1}}  \frac{\mu^2(d) d^\alpha}{\vphi(d)} \\
&\ = \frac{\sqrt{q_1} \Psi(x, y)}{\vphi(q_1) T^{\alpha/2}} \left(\frac{q}{q_1}\right)^{-\alpha}
\prod_{\substack{p | q/q_1 \\ p\nmid q_1}} \left(1+\frac{p^\alpha}{p-1}\right) \\
&\ \leq \frac{\sqrt{q_1} \Psi(x, y)}{\vphi(q) T^{\alpha/2}} \left(\frac{q}{q_1}\right)^{1-\alpha} 2^{\omega(q/q_1)}
\end{align*}
qui est bien la majoration voulue.
\end{proof}

\subsection{Estimation de $L(s, \chi ; y)$ dans la bande critique}

Une application du Lemme~\ref{perron_phi} fournit lorsque~$z\not\in\bfN$
\begin{equation}\label{int_perron_full}
\Psi_0(z, y ; \chi, \gamma) =  \frac{1}{2i\pi} \int_{\kappa-i\infty}^{\kappa+i\infty} L(s, \chi ; y) z^s \cPhi(\gamma x, s) \dd s
\end{equation}
où l'intégrale converge en valeur principale.
Le lemme suivant, repris pour l'essentiel de~\cite[Lemma~1]{HarperBV2012}, fournit un contrôle sur les variations de~$L(s, \chi ; y)$.
La qualité de cette estimation est étroitement liée à notre connaissance d'une région sans zéro pour~$L(s, \chi)$.

Dans cette section et les suivantes, $c_1$ et $c_2$ désignent toujours des constantes absolues positives,
$c_1$ étant choisie typiquement grande et~$c_2$ typiquement petite.


\begin{lemme}\label{majo_l_partiel}
Il existe des constantes~$c_1, c_2 >0$ telles que lorsque~$\chi$ est un caractère primitif de module~$q > 1$,
$\ee \in ]0, 1/2]$, $H \geq 4$ et lorsque la fonction~$L(s, \chi)$ n'a pas de zéro dans la région
\begin{equation}\label{l_partiel_rsz} \{s \in \bfC \mid \sigma \in ]0,\ee]\cup[1-\ee, 1], \tau \in[-H,H] \}, \end{equation}
alors pour tout~$y\geq (qH)^{c_1}$ et tout~$s\in \bfC$ avec~$\sigma \in [0, 1[$ et~$|\tau| \leq H/2$, on ait
\begin{equation}\label{estim_l_partiel_norm}
\begin{aligned}
\sum_{n \leq y} \frac{\Lambda(n) \chi(n)}{n^s} =
O\left( \frac{y^{1-\sigma-c_2\ee}}{1-\sigma} + \frac{y^{1-\sigma} \log^2(qyH)}{(1-\sigma)H} + \log(qH) + \frac{1}{\ee} \right)
.\end{aligned}
\end{equation}
Si~$\chi$ est réel et si~$L(s, \chi)$ a dans la région~\eqref{l_partiel_rsz} un unique zéro~$\beta$, qui est réel, alors
\begin{equation}\label{estim_l_partiel_sieg}
\begin{aligned}
\sum_{n \leq y} \frac{\Lambda(n) \chi(n)}{n^s} = - \frac{y^{\beta-s}-1}{\beta-s} + 
O\left( \frac{y^{1-\sigma-c_2\ee}}{1-\sigma} + \frac{y^{1-\sigma} \log^2(qyH)}{(1-\sigma)H} + \log(qH) + \frac{1}{\ee} \right)
.\end{aligned}
\end{equation}
Enfin, si la fonction~$\zeta$ n'a pas de zéro dans la région~\eqref{l_partiel_rsz}, alors
\begin{equation}\label{estim_l_partiel_pcp}
\begin{aligned}
\sum_{n \leq y} \frac{\Lambda(n)}{n^s} = \frac{y^{1-s}-1}{1-s} + 
O\left( \frac{y^{1-\sigma-c_2\ee}}{1-\sigma} + \frac{y^{1-\sigma} \log^2(yH)}{(1-\sigma)H} + \log H + \frac{1}{\ee} \right),
\end{aligned}
\end{equation}
\end{lemme}

\begin{proof}
L'estimation~\eqref{estim_l_partiel_norm} découle d'un cas particulier de~\cite[Lemma~1]{HarperBV2012}.
L'estimation~\eqref{estim_l_partiel_pcp} est à rapprocher de~\cite[Lemma~8]{TeneHild86}.
Les cas complémentaires n'apportent pas de difficulté essentielle.
Par souci de complétude on en reprend ici la démonstration, qui suit celle de Harper~\cite[Lemma~1]{HarperBV2012}.
Afin d'unifier les calculs dans les différents cas, on se donne~$\chi$ un caractère primitif de module~$q\geq 1$ qui peut être le caractère
trivial, et suivant les cas :
\begin{itemize}
\item lorsque~$\chi=\bfUn$, on note~$\theta(\chi) := -1$ et~$\beta_\chi := 1$,
\item sinon, si~$L(s, \chi)$ ne s'annule pas dans la région~\eqref{l_partiel_rsz}, on pose~$\theta(\chi):=0$,
\item enfin, si~$\chi$ est réel et si~$L(s, \chi)$ s'annule une seule fois dans la région~\eqref{l_partiel_rsz} en~$s=\beta$,
on pose~$\theta(\chi):=1$ et~$\beta_\chi := \beta$.
\end{itemize}
La quantité~$\beta_\chi$ n'interviendra pas dans les calculs lorsque~$\theta(\chi)=0$. On note
\[ S_s(y) := \sum_{n \leq y} \Lambda(n) \chi(n) n^{-s} . \]
La majoration triviale~$S_s(y) \ll y^{1-\sigma}/(1-\sigma)$ montre que l'on peut supposer~$\ee \geq 1/\log y$.
D'autre part, sans perte de généralité on suppose que~$s$ n'est pas un zéro de~$L(s, \chi)$ et est différent de~$0$.

On note~$F(s, \chi) := L(s, \chi)(s-\beta_\chi)^{-\theta(\chi)}$ et on rappelle les faits suivants,
énoncés dans~\cite[chapitres~15 et~16]{davenport2000} :
\begin{itemize}
\item $F$ est une fonction entière de~$s$ dont les seuls zéros sont d'une part les zéros triviaux, qui sont des entiers négatifs ou nuls,
et les zéros non triviaux, de parties réelles dans~$[0, 1]$,
\item le nombre de zéros~$\rho = \beta+i\gamma$ de~$F$ avec~$\beta \in[0,1]$ et~$|\gamma|\leq T$ vaut
\[ \frac{T}{\pi}\log\left(\frac{qT}{2\pi}\right) - \frac{T}{\pi} + O(\log(qT)) ,\]
\end{itemize}

Enfin, si~$\chi$ est non trivial et~$\chi(-1)=1$, on pose~$\alpha(\chi)=1$, et~$\alpha(\chi)=0$ dans tous les autres cas. Ainsi,
$\alpha(\chi)=1$ si et seulement si~$L(0, \chi)=0$.
Une formule de Perron~\cite[Corollaire II.2.4]{Tene2007} ainsi que des estimations classiques
concernant la densité verticale des zéros de~$L(s, \chi)$ (voir par exemple~\cite[chapitres~17 et~19]{davenport2000}) fournissent
\begin{equation}\label{perron_S}
\begin{aligned}
S_s(y) = &\ -\sum_{\substack{\rho \\ |\Im(\rho)-\tau| \leq H/2}} \frac{y^{\rho-s}}{\rho-s} - \theta(\chi) \frac{y^{\beta_\chi-s}-1}{\beta_\chi - s}
+ \alpha(\chi)\frac{y^{-s}}{s} - \frac{F'}{F}(s, \chi) \\
&\ + O\left(y^{-\sigma} + \frac{y^{1-\sigma} \log^2(q y H)}{H}\right)
\end{aligned}
\end{equation}
où~$\rho$ dans la première somme désigne un zéro non trivial de~$F(s, \chi)$.

On suppose dans un premier temps~$1-\sigma \leq \ee/2$. Alors
\begin{equation}\label{S_majo_part1}
S_s(y) + \theta(\chi)\frac{y^{\beta_\chi-s}-1}{\beta_\chi-s} \ll \sum_{\substack{\rho = \beta+i\gamma \\ |\gamma| \leq H}} \frac{y^{\beta-\sigma}}{|\rho-s|}
+ \left|\frac{F'}{F}(s, \chi)\right| + 1 + \frac{y^{1-\sigma} \log^2(q y H)}{H}
.\end{equation}
On a
\[ \frac{F'}{F}(s, \chi) \ll \left| \frac{F'}{F}(1+\ee + i\tau, \chi) \right| + \ee \max_{\sigma \leq \kappa \leq 1+\ee}
\left|\bigg(\frac{F'}{F}\bigg)'(\kappa+i\tau, \chi)\right| .\]
En dérivant une formule explicite pour $L'/L(s, \chi)$ (voir par exemple~\cite[chapitre~12, formule~(17)]{davenport2000}), on obtient
\begin{equation}\label{expr_somme_zeros}
\left(\frac{F'}{F}\right)'(\kappa+i\tau,\chi) = -\sum_{\rho}{\frac{1}{(\kappa+i\tau-\rho)^2}}
- \sum_{\substack{m\in \bfZ, m\leq 0 \\ L(m, \chi) = 0}} \frac{1}{(\kappa+i\tau-m)^2}
\end{equation}
où, dans la première somme, $\rho$ désigne un zéro non trivial de~$L(s, \chi)$, sauf éventuellement~$\beta_\chi$.
On a~$\kappa \gg 1$, la seconde somme est donc~$O(1)$. Dans la première somme sur~$\rho = \beta + i\gamma$,
\begin{itemize}
\item la contribution de ceux vérifiant~$|\gamma| > H$ est
\[ \ll \sum_{\substack{\rho \\ |\gamma|> H}} \frac{1}{|\gamma|^2} \ll \frac{\log(qH)}{H} \]
grâce par exemple à~\cite[formules~(1) des chapitre~15 et~16]{davenport2000},
\item la contribution de ceux vérifiant~$|\gamma| \leq H$ et~$|\tau-\gamma|>1$ est~$O(\log(qH))$
grâce à~\cite[formules~(3) des chapitres~15 et~16]{davenport2000},
\item la contribution de ceux vérifiant~$|\tau - \gamma| \leq 1$ est
\begin{align*}
\leq \sum_{\substack{\rho \\ |\gamma - \tau| \leq 1}} \frac{1}{|1+\ee + i\tau - \rho|^2} 
\leq \frac{1}{\ee} \Re\Bigg( \sum_{\substack{\rho \\ |\gamma - \tau| \leq 1}} \frac{1}{1+\ee + i\tau - \rho} \Bigg)
\ll \frac{1}{\ee^2} + \frac{\log(qH)}{\ee}
\end{align*}
en suivant les mêmes calculs que Harper~\cite[démonstration du lemma~3]{harper2012paper} et en notant que dans le cas~$\theta(\chi) \neq 0$,
on a~$1/(1+\ee+i\tau - \beta_\chi) \ll 1/\ee$.
\end{itemize}
On obtient donc $F'/F(s, \chi) \ll \ee^{-1} + \log(qH)$. Il reste à majorer la somme sur~$\rho$ du membre de droite de~\eqref{S_majo_part1}.
On utilise pour cela la majoration suivante, qui découle de~\cite[formule~(1.1)]{huxley1974} et~\cite[formule~(1.8)]{jutila1977},
\begin{equation}\label{log_free_bound} \card\Big\{ \rho = \beta + i\gamma \in \bfC\ \Big| \ \prod_{r\leq q} \prod_{\chi' \mod{r}} L(\rho, \chi') = 0, \beta \geq 1-\delta, |\gamma| \leq H \Big\}
\ll (qH)^{c_3 \delta} \end{equation}
pour une certaine constante~$c_3>0$, uniformément pour~$\delta\in[0, 1/2]$. On a ainsi
\begin{align*}
\sum_{\substack{\rho = \beta+i\gamma \\ |\gamma| \leq H}} \frac{y^{\beta-\sigma}}{|\rho-s|}
\ll &\ \sum_{\substack{\rho = \beta+i\gamma \\ \beta \leq 1/2, |\gamma| \leq H}} \frac{y^{1/2-\sigma}}{1+|\gamma-\tau|} +
\sum_{k=1}^{\floor{1/(2\ee)}} \sum_{\substack{\rho = \beta+i\gamma \\ k\ee \leq 1-\beta < (k+1)\ee \\ |\gamma| \leq H}} \frac{y^{1-\sigma-k\ee}}{\ee} \\
\ll &\ y^{1/2-\sigma} \log^2(qH) + \frac{y^{1-\sigma-c_2 \ee}}{1-\sigma}
\end{align*}
pour une certaine constante~$c_2>0$, quitte à supposer~$c_1>c_3$. Cela fournit la majoration annoncée dans le cas~$1-\sigma \leq \ee/2$.

Dans le cas~$1-\sigma > \ee/2$, on a par une intégration par parties
\begin{equation}\label{S_ipp}
S_s(y) = S_s(\sqrt{y}) + S_{i\tau}(y)y^{-\sigma} - S_{i \tau}(\sqrt{y})y^{-\sigma/2}  + \sigma \int_{\sqrt{y}}^y S_{i\tau}(t) t^{-\sigma-1} \dd t
.\end{equation}
Soit~$t\in[\sqrt{y}, y]$ ; on a~$t \geq (qH)^{c_1/2}$. Il est nécessaire de distinguer le cas~$\theta(\chi)=1$ car alors~$F(1-\beta_\chi, \chi)=0$.
On note donc
\[ \bfUn_{\theta=1} :=
\begin{cases} 1 & \mbox{ si } \theta(\chi)=1 \\
0 & \mbox { sinon.}
\end{cases} \] On suppose également sans perte de généralité que~$\tau\neq 0$.
Il découle de la formule~\eqref{perron_S} avec~$s$ et~$y$ remplacés respectivement par~$i\tau$ et~$t$ que
\begin{equation}\label{majo_S_part2}
\begin{aligned}
&\ S_{i\tau}(t) + \theta(\chi)\frac{t^{\beta_\chi-i\tau}-1}{\beta_\chi-i\tau} \\
\ll &\ \sum_{\substack{\rho = \beta+i\gamma \\ \rho \neq 1-\beta_\chi \\ |\gamma| \leq H}} \frac{t^{\beta}}{|\rho-i\tau|}
+ \left|\alpha(\chi)\frac{t^{-i\tau}}{i\tau} - \frac{F'}{F}(i\tau, \chi)
- \bfUn_{\theta=1} \frac{t^{1-\beta_\chi-i\tau}}{1-\beta_\chi-i\tau} \right|
+ 1 + \frac{t \log^2(q t H)}{H}
\end{aligned}
\end{equation}
où dans la somme sur~$\rho$ la condition~$\rho \neq 1-\beta_\chi$ n'est à prendre en compte que lorsque~$\theta(\chi)=1$.
D'après~\cite[formules~(10.27), (12.9) et Theorem~11.4]{montgomery2006multiplicative}, on a
\begin{align*}
&\ \left| \alpha(\chi)\frac{t^{-i\tau}}{i\tau} - \frac{F'}{F}(i\tau, \chi)
- \bfUn_{\theta=1} \frac{t^{1-\beta_\chi-i\tau}}{1-\beta_\chi-i\tau} \right| \\
\leq &\ \left| \frac{L'}{L}(1-i\tau, \overline{\chi}) - \frac{\bfUn_{\theta=1}}{1-i\tau-\beta_\chi} \right|
+ \left|\bfUn_{\theta=1} \frac{t^{1-i\tau-\beta_\chi}-1}{1-i\tau-\beta_\chi} \right|
+ \left|\alpha(\chi)\frac{t^{-i\tau}-1}{i\tau}\right| + O(\log(qH)) \\
\ll &\ \log(q y H) + \sqrt{t}\log t.
\end{align*}
Enfin, pour tout~$\rho=\beta+i\gamma$ zéro non trivial de~$F(s, \chi)$, sauf éventuellement~$1-\beta_\chi$, on a~$\beta\geq \ee$, ainsi
\begin{align*}
\sum_{\substack{\rho = \beta+i\gamma \\ \rho \neq 1-\beta_\chi \\ |\gamma| \leq H}} \frac{t^{\beta}}{|\rho-i\tau|}
\ll &\ \sum_{\substack{\rho = \beta+i\gamma \\ |\gamma| \leq H \\ \beta \leq 1/4}} \frac{t^{1/4}}{\ee+|\gamma-\tau|} +
\sum_{\substack{\rho = \beta+i\gamma \\ |\gamma| \leq H \\ 1/4 < \beta < 1/2}} \frac{t^{1/2}}{1+|\gamma-\tau|} +
\sum_{k=1}^{\floor{1/(2\ee)}} \sum_{\substack{\rho = \beta+i\gamma \\ |\gamma| \leq H \\ k\ee \leq 1-\beta < (k+1)\ee}} t^{1-k\ee} \\
\ll &\ \sqrt{t} \log^2(qH) + t \left(\frac{(qH)^{c_3}}{t}\right)^{\ee} \ll t^{1-c_2 \ee}
\end{align*}
en supposant~$c_1>2c_3$ et quitte à réduire la valeur de~$c_2$, et où l'on a de nouveau utilisé
des résultats classiques sur la densité des zéros de~$L(s, \chi)$~\cite[formules~(1) des chapitres~15 et~16]{davenport2000}.
On a donc
\[ S_{i\tau}(t) + \theta(\chi)\frac{t^{\beta_\chi-i\tau}-1}{\beta_\chi-i\tau} \ll t^{1-c_2\ee} + \log(qyH) + \frac{t \log^2(qyH)}{H} \]
et ainsi, en reportant dans~\eqref{S_ipp},
\begin{align*}
&\ S_s(y) + \theta(\chi)\frac{y^{\beta_\chi-s}-1}{\beta_\chi-s} \\
\ll &\ \left|\theta(\chi) \frac{y^{(\beta_\chi-s)/2}-1}{\beta_\chi-s}\right| +  \frac{y^{(1-\sigma)/2}}{1-\sigma}
+ \frac{y^{1-\sigma-c_2\ee}}{1-\sigma-c_2\ee} + y^{-\sigma/2} \log(qyH)
 + \frac{y^{1-\sigma} \log^2(qyH)}{(1-\sigma) H}
.\end{align*}
Dans le membre de droite, le premier terme est dominé par le deuxième.
En utilisant l'inégalité~$1-\sigma > \ee/2$ et en observant que~$y^{(1-\sigma)/2}/(1-\sigma) \gg \log y$, on obtient la majoration annoncée.
\end{proof}
\begin{remarque}
Ainsi qu'il est observé dans la remarque qui suit le lemme~2 de~\cite{HarperBV2012}, dans la démonstration qui précède,
la majoration~\eqref{log_free_bound} en conjonction avec l'hypothèse~$y \geq (qH)^{c_1}$, remplace avantageusement les résultats classiques
sur la densité verticale des zéros des fonctions~$L$~\cite[chapitres~17 et~19]{davenport2000}.
L'utilisation de ceux-ci induirait un facteur supplémentaire~$\log^2(qyH)$ dans le premier terme d'erreur
de chacune des estimations~\eqref{estim_l_partiel_norm}, \eqref{estim_l_partiel_sieg} et \eqref{estim_l_partiel_pcp}
et rendrait celles-ci triviales lorsque~$\ee = O((\log\log qyH)/ \log y)$.
Des valeurs permises pour~$\ee$ dépendent le choix des paramètres~$\ee$ et~$T$ dans les Propositions~\ref{caracts_norm},
\ref{caracts_pcp} et~\ref{caracts_sieg} \emph{infra}, qui influent sur le domaine de validité en~$Q$ ainsi que la qualité des termes d'erreur.
\end{remarque}


\begin{lemme}\label{majo_l_decal}
Il existe des constantes~$c_1, c_2 > 0$ positives, $c_2$ pouvant être fixée arbitrairement petite,
telles que pour tous réels~$x, y, T$ supérieurs à~$4$, $\ee>0$ et tout entier~$q\geq 2$, sous les conditions~:
\begin{itemize}
\item $(\log x)^{c_1} \leq y \leq x$,
\item $qT \leq y^{c_2}$,
\item $\ee \log y \geq 1/c_2$,
\item $T \geq y^{c_1 \ee} (\log x)^2$,
\end{itemize}
et pour tout caractère~$\chi$ de module~$q$ tel que la fonction~$L(s, \chi)$ ne s'annule pas
pour~$\sigma \geq 1-\ee$ et~$|\tau|\leq 2T$, la majoration
\[ \frac{L(\sigma+i\tau, \chi ; y)}{L(\sigma'+i\tau, \chi ; y)} \ll x^{(\sigma'-\sigma)/2} \]
soit valable lorsque~$(\sigma, \sigma', |\tau|) \in [\alpha-c_2 \ee, \alpha]^2 \times [0, T]$ et~$\sigma\leq \sigma'$.

En particulier, cette majoration est valable avec~$\ee = b/\log QT$ pour tout~$Q \geq 2$ vérifiant~$QT \leq y^{c_2}$, lorsque~$\chi$
est un caractère de module~$q\leq Q$ qui n'est pas~$Q$-exceptionnel.
De plus, elle est également valable lorsque~$\chi$ est~$Q$-exceptionnel et l'une des deux conditions suivantes est vérifiée :
\[ |\tau| \geq \max\{1, y^{\beta-\sigma}\} \qquad \text{ ou } \qquad \beta \leq 1-\sqrt{c_2}/\log QT .\]

D'autre part, sous les conditions :
\begin{itemize}
\item $(\log x)^{c_1} \leq y \leq x$,
\item $y^{c_1(\log T)^{-2/3}(\log \log T)^{-1/3}} (\log x)^2 \leq T \leq y^{c_2}$,
\end{itemize}
la majoration
\[ \frac{\zeta(\sigma+i\tau, y)}{\zeta(\alpha+i\tau, y)} \ll x^{(\alpha-\sigma)/2} \]
est valable lorsque~$(\sigma, |\tau|) \in [\alpha-c_2 (\log T)^{-2/3}(\log \log T)^{-1/3}, \alpha] \times [y^{1-\alpha}, T]$.
\end{lemme}

\begin{proof}
Afin d'unifier les calculs dans les différents cas, on se donne un caractère~$\chi$, qui est soit non principal et de module~$q\geq 2$,
soit le caractère trivial auquel cas l'on pose~$q := 1$, et on note suivant les cas :
\begin{itemize}
\item si~$\chi = \bfUn$, on pose~$\beta_\chi := 1$, $\sigma'=\alpha$ et~$\ee = b(\log T)^{-2/3}(\log \log T)^{-1/3}$,
\item si~$\chi$ est un caractère $Q$-exceptionnel, on pose~$\beta_\chi := \beta$ et~$\ee = b/ \log QT$.
\end{itemize}
Ainsi~$L(s, \chi)$ est une fonction qui n'a pas de zéro ni de pôle pour~$\sigma\geq 1-\ee$ et~$|\tau|\leq 2T$,
sauf éventuellement en~$s=\beta_\chi$. Dans le cas~$\chi=\bfUn$, ceci découle de la région sans zéro
de Vinogradov-Korobov~\cite[formule~(6.24)]{montgomery2006multiplicative} quitte à réduire la valeur de~$b$.
Quitte à choisir~$c_1$ suffisamment grande et~$c_2$ suffisamment petite on a~$\sigma' \geq \sigma \geq 1/2$.
On note~$\chi^*$ le caractère primitif associé à~$\chi$ et~$q^*$ son module ; on a pour tout~$s\in \bfC$,
\[ L(s, \chi ; y) = \prod_{p|q}(1-\chi^*(p)p^{-s}) L(s, \chi^* ; y) .\]
Lorsque~$q=1$ et~$\chi=\bfUn$, le produit sur~$p$ est vide,
et dans les autres cas, sa dérivée logarithmique par rapport à~$s$
est~$\ll \sum_{p|q}(\log p)/(1-p^{-\Re(s)}) \ll \log q$ lorsque~$\Re(s) \geq 1/2$.
On a donc dans tous les cas
\[ \frac{L(\sigma+i\tau, \chi ; y)}{L(\sigma'+i\tau, \chi ; y)} =
\exp\left\{ -\int_\sigma^{\sigma'} \frac{L'}{L}(\kappa+i\tau, \chi^* ; y) \dd \kappa + O(1) \right\} \]
avec, pour tout~$\kappa \in [\sigma, \sigma']$,
\[ -\frac{L'}{L}(\kappa + i\tau, \chi^* ; y) = \sum_{P(n) \leq y} \frac{\Lambda(n)\chi^*(n)}{n^{\kappa+i\tau}}
= \sum_{n\leq y}\frac{\Lambda(n)\chi^*(n)}{n^{\kappa+i\tau}} + O(1) .\]
Le Lemme~\ref{majo_l_partiel} s'applique avec~$H=2T$.
Lorsque~$\chi$ est non exceptionnel, le Lemme~\ref{majo_l_partiel} fournit
\begin{align*}
\sum_{n \leq y} \frac{\Lambda(n)\chi^*(n)}{n^{\kappa+i\tau}}
\ll &\ \frac{y^{1-\kappa-c_3\ee}}{1-\kappa} + \frac{y^{1-\kappa} \log^2(qyT)}{(1-\kappa)T} + \log(qT) \\
\ll &\ \frac{y^{1-\alpha-c_3\ee/2}}{1-\alpha} + \log(qT)
\end{align*}
pour une certaine constante~$c_3>0$, quitte à supposer~$c_1$ suffisamment grande et~$c_2$ suffisamment petite.
Si~$\chi$ est exceptionnel ou~$\chi=\bfUn$, on a
\[ \sum_{n \leq y} \frac{\Lambda(n)\chi^*(n)}{n^{\kappa+i\tau}}
\ll \frac{y^{1-\alpha-c_3\ee/2}}{1-\alpha} + \log(qT) + \left|\frac{y^{\beta_\chi-\kappa-i\tau}-1}{\beta_\chi-\kappa-i\tau}\right| \]
Lorsque~$|\tau| \geq y^{\beta_\chi-\sigma}$, le dernier terme du membre de droite est borné,
tandis que lorsque~$\chi$ est $Q$-exceptionnel et~$\beta \leq 1-\sqrt{c_2}/\log QT$,
ce terme est~$O(y^{1-\kappa-\sqrt{c_2}\ee/b}/(1-\kappa))$.

Ainsi dans tous les cas, quitte à réduire la valeur de~$c_2$, on a
\begin{align*}
\sum_{n \leq y} \frac{\Lambda(n)\chi^*(n)}{n^{\kappa+i\tau}} &\ \ll \frac{y^{1-\alpha-\sqrt{c_2}\ee/b}}{1-\alpha} + \log(QT) \\ &\ \ll
\begin{cases} \left(\e^{-\sqrt{c_2} \ee \log y / b} + \frac{\log(QT)}{\log x}\right)\log x & \text{ si~$\chi \neq \bfUn$ } \\
\left(\e^{-c_2^{-1/6} (\log y)^{1/3}/(\log\log y)^{1/3}} + \frac{\log(QT)}{\log x}\right)\log x & \text{ si~$\chi = \bfUn$ } 
\end{cases}
\end{align*}
grâce à~\cite[formule~(III.5.74)]{Tene2007}. Quitte à supposer~$c_1$ suffisamment grande et~$c_2$ suffisamment petite, on en déduit
\[ \left| \sum_{P(n) \leq y} \frac{\Lambda(n)\chi^*(n)}{n^{\kappa+i\tau}} \right| \leq \frac{\log x}{2} + O(1) \]
donc~$L(\sigma+i\tau, \chi ; y) / L(\sigma'+i\tau, \chi ; y) = O(x^{(\sigma'-\sigma)/2})$.
\end{proof}


Le lemme suivant traite de la situation où le zéro exceptionnel existe.
La démonstration est analogue à celle de~\cite[lemme~1]{tenenbaum1985probleme}.
\begin{lemme}\label{estim_lsieg}
Il existe des constantes~$c_1, c_2$ strictement positives telles que pour tous réels~$Q, T$ supérieurs à~$2$
et~$x$ et~$y$ assez grands avec :
\begin{itemize}
\item $(\log x)^{c_1} \leq y \leq x$,
\item $QT \leq  y^{c_2/(\log\log\log x)}$,
\item $T \geq y^{c_1/\log(QT)} (\log x)^2$,
\end{itemize}
si le zéro exceptionnel~$\beta$ existe et vérifie~$1-\beta \leq \sqrt{c_2}/\log QT$,
alors pour tout~$\tau$ avec~$|\tau| \leq T/2$ on ait
\[ L(\alpha+i\tau+\beta-1, \chi_1 ; y) \ll \zeta(\alpha, y) H(u)^{-\delta} .\]
\end{lemme}

\begin{proof}
Quitte à choisir~$c_1$ suffisamment grande, on suppose~$\alpha \geq 2/3$. Alors on a
\begin{align*}
L(\alpha+i\tau+\beta-1, \chi_1 ; y) = &\ \zeta(\alpha, y)
\exp\left\{ \sum_{p \leq y} \log\left(\frac{1-p^{-\alpha}}{1-\chi_1(p)p^{1-\beta-\alpha-i\tau}}\right) \right\} \\
= &\ \zeta(\alpha, y) \exp\left\{ -\sum_{p \leq y}\frac{1 - \chi_1(p)p^{1-\beta-i\tau}}{p^\alpha} + O(1) \right\}
\end{align*}
le logarithme étant pris en détermination principale. La somme sur~$p$ vérifie la minoration
\[ \sum_{p \leq y}\frac{1 - \chi_1(p)p^{1-\beta}\cos(\tau\log p)}{p^\alpha}
\geq O(1) + \frac{1}{\log y} \sum_{n\leq y} \frac{\Lambda(n)(1 - \chi_1(n)n^{1-\beta}\cos(\tau\log n))}{n^\alpha} .\]
Le Lemme~\ref{majo_l_partiel} appliqué deux fois avec~$\ee = b/\log QT$, $H=2T$ et~$s \in\{\alpha,\alpha+\beta-1\}$
fournit pour une certaine constante~$c_3>0$
\[ \sum_{n\leq y}\frac{\Lambda(n)}{n^\alpha} = \frac{y^{1-\alpha}}{1-\alpha} +
O\left( \frac{y^{1-\alpha - c_3\ee}}{1-\alpha} + \frac{y^{1-\alpha}\log^2(QyT)}{(1-\alpha)T} + \log(QT) \right) \]
\[ \sum_{n\leq y}\frac{\Lambda(n)\chi_1(n)}{n^{\alpha+i\tau+\beta-1}} = -\frac{y^{1-\alpha-i\tau}}{1-\alpha-i\tau}
 + O\left( \frac{y^{2-\alpha-\beta-c_3\ee}}{2-\alpha-\beta} + \frac{y^{2-\alpha-\beta}\log^2(QyT)}{(2-\alpha-\beta)T} + \log(QT) \right) .\]
Quitte à choisir~$c_1$ suffisament grande, dans les termes d'erreur, le deuxième terme est dominé par le premier. On a
\[ \Re\left\{\frac{y^{1-\alpha}}{1-\alpha} + \frac{y^{1-\alpha-i\tau}}{1-\alpha-i\tau}\right\} \gg \frac{\log x}{(\log (u+1))^2} \]
grâce aux calculs de~Hildebrand et~Tenenbaum~\cite[Lemma~8]{TeneHild86}.
D'autre part, quitte à supposer~$c_2$ suffisament petite, on a~$1-\beta \leq c_3\ee/2$,
or~$y^{-c_3\ee/2} \leq (\log \log x)^{-c_3b/(2c_2)}$ et~$\log(QT) \ll \log x/ (u \log\log\log x)$,
ainsi pour un certain~$\delta>0$ et~$x$ et~$y$ assez grands, quitte à choisir~$c_2$ suffisamment petite on obtient
\[ L(\alpha+i\tau+\beta-1, \chi_1 ; y) \ll \zeta(\alpha, y) \exp\left\{ -\delta\frac{u}{(\log (u+1))^2} \right\} \]
qui est la majoration annoncée.
\end{proof}

\subsection{Caractères non principaux, non exceptionnels}

On s'intéresse au cas des caractères non principaux et non associés à l'éventuel caractère exceptionnel $\chi_1$.
Soit $\chi$ un tel caractère, de module $q$. On rappelle que~$\Psi_0(z, y ; \chi, \gamma)$ est la fonction définie par~\eqref{def_Psi0}.

\begin{prop}\label{caracts_norm}
Il existe des constantes~$c_1$ et~$c_2$ strictement positives telles que pour tous réels~$x, y, \gamma, \ee$ et~$T$
avec~$4 \leq T \leq y^{c_2}$ et~$\ee>0$, lorsque~$q$ et un entier avec~$2\leq q \leq y^{c_2}$ et $\chi$ un caractère de module~$q$, non principal
et tel que la fonction~$L(s, \chi)$ ne s'annule pas dans la région
\[ \{ s \in \bfC \mid \sigma > 1-\ee, |\tau| \leq T \} \]
et lorsque~$2\leq (\log x)^{c_1} \leq y \leq x$, et $z\in[x^{2/3}, x]$ on ait
\[ \Psi_0(z, y ; \chi , \gamma) \ll z^\alpha \zeta(\alpha, y) (1+|\gamma|z) \left( (\log T) x^{-c_2 \ee} + T^{-c_2} \right) .\]
En particulier, pour tout~$Q$ avec~$2 \leq Q \leq y^{c_2}$, ceci est valable pour touts les caractères de module inférieurs à~$Q$
qui sont non principaux et non~$Q$-exceptionnels, avec~$\ee = b/\log QT$.
De plus la même majoration est valable lorsque~$\chi$ est~$Q$-exceptionnel mais~$\beta \leq 1-\sqrt{c_2}/\log QT$.
\end{prop}

\begin{remarque} Lorsque~$\gamma=0$, ce résultat est un cas particulier de~\cite[theorem~3]{HarperBV2012}.\end{remarque}

\begin{proof}
La condition sur~$(x, y)$ assure que les hypothèses du Lemme~\ref{perron_phi}
sont vérifiées pour la suite~$a_n = \e(n\gamma) \chi(n)$ avec~$M$ absolu. On a de plus~$\alpha \geq 1/2$ quitte à supposer~$c_1$ assez grande.
Le choix~$\kappa = \alpha(x, y)$ fournit pour une certaine constante~$c_3>0$
\begin{align*}
\Psi_0(z, y ; \chi, \gamma) =&\ \frac{1}{2 i\pi} \int_{\alpha-iT}^{\alpha+iT}{L(s, \chi ; y) z^s \cPhi(\gamma z, s)\dd s}
+ O( z^\alpha \zeta(\alpha, y) (1+|\gamma|z) T^{-c_3} )
.\end{align*}
On modifie le contour pour intégrer sur la ligne brisée passant par les points
\[ \alpha-iT, \quad \alpha-c_4\ee-iT, \quad  \alpha-c_4\ee+iT, \quad \alpha+iT \]
où~$c_4$ est une constante absolue choisie plus petite que la constante~$c_2$ du Lemme~\ref{majo_l_decal}.
Soit $I_1$ la contribution des deux segments horizontaux et $I_2$ la contribution du segment vertical.
Les Lemmes~\ref{majo_phi} et~\ref{majo_l_decal} s'appliquent ici dans tous les cas envisagés dans l'énoncé. On a ainsi
\begin{align*}
I_1 &\ \ll z^\alpha \zeta(\alpha, y) \frac{(1+|\gamma|z)}{T} \int_{0}^{c_4\ee} \left(\frac{\sqrt{x}}{z}\right)^{\kappa} \dd \kappa  \\
&\  \ll z^\alpha \zeta(\alpha, y) \frac{(1+|\gamma|z)}{T},
\end{align*}
\begin{align*}
I_2 &\ \ll z^\alpha \zeta(\alpha, y) (1+|\gamma|z) (\log T) \left(\frac{\sqrt{x}}{z}\right)^{c_4\ee} \\
&\ \ll z^\alpha \zeta(\alpha, y) (1+|\gamma|z) (\log T) x^{-c_4\ee/6}.
\end{align*}
On obtient la majoration annoncée en regroupant ces deux estimations.
\end{proof}


Dans la somme du membre de droite de~\eqref{E_sum_chi}, la contribution des caractères principaux s'écrit,
après interversion des sommes,
\[ V(x, y ; q, \eta) := \sum_{n \in S(x, y)}{\frac{\mu(q / (q, n))}{\vphi(q / (q, n))} \e(n \eta)} \]
et celle, lorsque $\DSZ(q) = 1$, des caractères associés à $\chi_1$ s'écrit~$\chi_1(a) W(x, y ; q, \eta)$ où
\[ W(x, y ; q, \eta) := \frac{\tau(\chi_1)}{\vphi(q_1)}
\sum_{ r | q/q_1 } {\frac{\mu(r)\chi_1(r)}{\vphi(r)} \sum_{m \in S(x q_1 r/q, y)}{ \e\left(\frac{m q}{q_1 r} \eta\right) \chi_r(m) } } .\]
On rappelle que~$T_1$ est défini en~\eqref{defs_t1_t2}.

\begin{prop}\label{prop_E_VW}
Il existe deux constantes positives~$c_1$ et~$c_2$ telles que pour tous réels~$x$, $y$ et~$\vth$ et tous entiers~$q, Q\geq 1$,
lorsque~$(\log x)^{c_1} \leq y \leq x$, $q\leq y^{c_2}$, $Q \leq T_1^{c_2}$ et~$\vth = a/q+\eta$ avec~$(a, q)=1$, on ait
\begin{equation}\label{estim_E_VW}
\begin{aligned}
E(x, y ; \vth) = &\ V(x, y ; q, \eta) + \DSZ(q)\chi_{1}(a) W(x, y ; q, \eta) \\
&\ + O\left(\Psi(x, y) (1+|\eta|x) \left(y^{-c_2} + Q^{-c_2} \right) \right)
\end{aligned}
\end{equation}
où~$\chi_1$ désigne le caractère primitif $Q$-exceptionnel.
\end{prop}

\begin{proof}
La preuve que l'on propose ici reprend la structure des calculs de la section~3.3 de~\cite{HarperBV2012}.
Les caractères de modules inférieurs à~$Q$ sont traités grâce à la Proposition~\ref{caracts_norm}.
Pour les caractères de modules supérieurs, lorsque la fonction~$L$ a ses zéros de petite partie réelle, la Proposition~\ref{caracts_norm}
permet encore de conclure. Cela ne concerne pas tous les caractères de modules supérieurs à~$Q$, mais la majoration de Huxley et Jutila~\eqref{log_free_bound}
permet de dire que les caractères restants sont en proportion suffisament peu nombreux pour que leur contribution, même majorée trivialement, soit bien contrôlée.

Soient~$c'_1$ et~$c'_2$ les constantes de la Proposition~\ref{caracts_norm}. On suppose~$c_1\geq c'_1$ et~$c_2 \leq c'_2$.
Il s'agit de majorer
\begin{equation}\label{somme_cars_norm}
\sum_{d|q}{\frac{1}{\vphi(q/d)} {\mathop{{\sum}'}_{\substack{\chi \mod{q/d}}}}{\chi(a) \tau(\overline{\chi}) \Psi_0(x/d, y ; \chi, \eta d)} } 
\end{equation}
où la somme $\sum'$ porte sur les caractères non principaux et non $Q$-exceptionnels.
Lorsqu'un tel caractère est lui-même de module~$q/d \leq Q$, la Proposition~\ref{caracts_norm} s'applique
avec $z=x/d$ et $\gamma=\eta d$, $T = T_1$ et~$\ee=b/\log QT$ et fournit
\begin{align*}
\Psi_0(x/d, y ; \chi, \eta d) \ll x^\alpha \zeta(\alpha, y) d^{-\alpha} (1+|\eta|x) \left( y^{-c_3} + \cL^{-c_3} + (\log x)^{1/2} x^{-c_3/\log Q} \right)
\end{align*}
pour une certaine constante~$c_3>0$. On sépare la somme sur~$d$ dans~\eqref{somme_cars_norm} selon si~$q/d\leq Q$ ou pas.
La contribution des~$d|q$ avec~$d\geq q/Q$ est certainement
\begin{align*}
&\ \ll \Psi(x, y) (\log x) (1+|\eta|x) \left( y^{-c_3} + \cL^{-c_3} + (\log x)^{1/2} x^{-c_3/\log Q} \right) \sum_{r|q,\ r\leq Q} \sqrt{r} (q/r)^{-\alpha} \\
&\ \ll \Psi(x, y) (1+|\eta|x) \left( y^{-c_2} + \cL^{-c_2} \right)
\end{align*}
en majorant trivialement la somme sur~$r$ par~$O\big(\min\{q, Q\}^{3/2}\big)$, quitte à réduire la valeur de~$c_2$ et augmenter la valeur de~$c_1$
afin d'avoir~$\alpha \geq 2/3$ et pour absorber le facteur~$\log x$.
La dernière inégalité fait usage de l'hypothèse~$Q \leq T_1^{c_2}$, quitte à supposer~$c_2<1$.

Il reste à majorer la contribution à l'expression~\eqref{somme_cars_norm} des~$d|q$ avec~$d<q/Q$, autrement dit des caractères
de modules~$r|q$ avec~$Q < r \leq y^{c_2}$. Soit~$\chi$ un tel caractère et~$c'_3$
la constante apparaissant en exposant dans la formule~\eqref{log_free_bound}. On pose~$c_3 = 2/c'_2$ et~$c_4 = 1/(10(c_3+1)c'_3)$.
Quitte à diminuer la valeur de~$c_2$, on a~$r^{c_3} \leq y^{c'_2}$.
Lorsque~$L(s, \chi)$ ne s'annule pas pour~$\sigma \geq 1-c_4$ et~$|\tau| \leq r^{c_3}$, on a pour~$(\log x)^{c_1}\leq y \leq x$ et~$d|q$
la majoration
\[ \Psi_0(x/d, y ; \chi, \eta d) \ll \Psi(x, y) (1+|\eta|x) ( x^{-c'_2 c_4/2} + (\log x) r^{-2}) .\]
La contribution de tous ces caractères à la somme~\ref{somme_cars_norm} est donc
\[ \ll \Psi(x, y) (1+|\eta|x) ( x^{-c_2} + (\log x)Q^{-1/2}) .\]
Pour tout~$r \geq 2$, notons~$N_r$ le nombre de caractères de module~$r$ tel que la fonction~$L(s, \chi)$ s'annule au moins une fois
pour~$\sigma \geq 1-c_4$ et~$|\tau| \leq r^{c_3}$. La majoration~\eqref{log_free_bound} fournit~$\sum_{r\leq R} N_r \ll R^{1/10}$.
Pour un tel caractère, on a~$\tau(\overline{\chi})/\vphi(r) \ll r^{-1/3}$. La majoration triviale~$|\Psi(x/d, y ; \chi, \eta d)|\leq \Psi(x, y)$
montre que la contribution de tous ces caractères à la somme~\eqref{somme_cars_norm} est
\[ \ll \Psi(x, y) \sum_{Q< r \leq y^{c_2}} \frac{N_r}{r^{1/3}} \ll \Psi(x, y) (Q^{-c_2/5} + y^{-c_2/5}) \]
ce qui fournit la conclusion souhaitée.

\end{proof}
\begin{remarque}
Il est possible de montrer que cette estimation est valable pour~$q\leq x^{c_2}$, \emph{cf.}~\cite[Lemma~2]{HarperBV2012}.
Cela n'a cependant pas d'utilité pour les applications que l'on envisage ici.
\end{remarque}

\subsection{Caractères principaux par la méthode du col}

On note pour tout $s=\sigma+i\tau$ avec~$\sigma>0$ et tout entier~$q$ qui est $y$-friable
\[ \zeta(s, q ; y) := \sum_{P(n)\leq y}{\frac{\mu(q/(n,q))}{\vphi(q/(n,q))} n^{-s}} =
\frac{q^{1-s}}{\vphi(q)} \prod_{p|q}{\left( 1-p^{s-1} \right)} \zeta(s, y) .\]

Pour $\sigma\leq 1$, le facteur devant~$\zeta(s, y)$ est~$\ll 2^{\omega(q)} q^{1-\sigma}/\vphi(q)$.
On montre une première estimation de~$V(x, y ; q, \eta)$ par la méthode du col.
Par rapport à celle de la Proposition~\ref{caracts_pcp_saias},
qui sera montrée dans la section suivante, elle a l'avantage d'être valide sous des conditions moins restrictives sur~$q$ et~$\eta$,
au détriment du terme d'erreur.

\begin{prop}\label{caracts_pcp}
Il existe des constantes~$c_1$ et~$c_2$ positives
telles que pour tout $(x, y)$ avec~$(\log x)^{c_1} \leq y \leq \exp\{(\log x)/(\log \log x)^4\}$,
tout entier~$q\in S(x^{1/4}, y)$ et tout~$\eta\in\bfR$ vérifiant~$|\eta|x \leq x^{1/4}$, on ait
\begin{equation}\label{estim_V_col}
\begin{aligned}
V(x, y ; q, \eta) = &\ \frac{\alpha q^{1-\alpha}}{\vphi(q)} \prod_{p|q}{\left( 1-p^{\alpha-1} \right)}\cPhi(\eta x, \alpha) \Psi(x, y) \\
&\ + O\left(\frac{2^{\omega(q)}q^{1-\alpha} \Psi(x, y)}{\vphi(q)(1+|\eta|x)^\alpha}\frac{(\log q)^2\log(2+|\eta|x)^3}{u} \right) \\
&\ + O\left(\frac{2^{\omega(q)}q^{1-\alpha} \Psi(x, y)}{\vphi(q)} (1+|\eta|x)\left(y^{-c_2} + \e^{-c_2(\log x)^{3/5}(\log \log x)^{-1/5}} \right) \right)
.\end{aligned}
\end{equation}
\end{prop}

\begin{remarque}
En particulier, lorsque~$|\eta|x \leq \min\{y^{c_2/2}, \e^{c_2(\log x)^{3/5}(\log \log x)^{-1/5}/2} \}$, on a
\begin{equation}\label{majo_V_tv}
V(x,y;q,\eta) \ll \frac{2^{\omega(q)}q^{1-\alpha}\Psi(x,y)(\log q)^2 (\log(2+|\eta|x))^3}{\vphi(q)(1+|\eta|x)^\alpha}
.\end{equation}
\end{remarque}

\begin{proof}[Démonstration de la Proposition~\ref{caracts_pcp}]
Soit~$T$ un réel supérieur à~$4$. La série de Dirichlet~$\zeta(s, q ; y)$ est absolument convergente pour~$\sigma>0$, donc
les Lemmes~\ref{hypoth_perron_a} (avec~$q_1=1$) et~\ref{perron_phi} s'appliquent et fournissent
\begin{equation}\label{perron_v}
\begin{aligned}
V(x, y ; q, \eta) = &\ \frac{1}{2 i \pi} \int_{\alpha-iT}^{\alpha+iT}  \zeta(s, q ; y) x^s \cPhi(\eta x, s) \dd s \\
& + O\left(\frac{2^{\omega(q)}q^{1-\alpha} \Psi(x, y)}{\vphi(q)} \frac{(1+|\eta|x)\log x \log T}{T^{\alpha/2}} \right)
.\end{aligned}
\end{equation}
On note~$\ee := c_3(\log T)^{-2/3}(\log \log T)^{-1/3}$ pour un certain réel~$c_3>0$ fixé
plus petit que la constante~$c_2$ du Lemme~\ref{majo_l_decal}, et on choisit~$T = \min\{y^{c_2}, \e^{(\log x)^{3/5}(\log \log x)^{-1/5})}\}$.
Alors les hypothèses du Lemme~\ref{majo_l_decal} sont satisfaites quitte à choisir~$c_1$ assez grande et~$c_2$ assez petite.
D'autre part, quitte à augmenter la valeur de~$c_1$ et diminuer celle de~$c_2$, on suppose que~$\alpha - c_2\ee \geq 1/2$. 
On intègre suivant le chemin $\cup_{i=1}^{7}\cC_j$, où

\begin{enumerate}
\item $\cC_1$ est le segment $[\alpha-iT, \alpha-\ee -iT]$,
\item $\cC_2$ est le segment $[\alpha-\ee-iT, \alpha-\ee- i y^{1-\alpha+\ee}]$,
\item $\cC_3$ est le chemin reliant le point $\alpha-\ee -i y^{1-\alpha+\ee}$
au point $\alpha - i y^{1-\alpha}$ en suivant la courbe~$\tau = y^{1-\sigma}$,
\item $\cC_4$ est le segment $[\alpha - i y^{1-\alpha}, \alpha + i y^{1-\alpha}]$,
\item $\cC_5$ est le chemin reliant le point $\alpha + i y^{1-\alpha}$
au point $\alpha-\ee +i y^{1-\alpha+\ee}$ en suivant la courbe~$\tau = -y^{1-\sigma}$,
\item $\cC_6$ est le segment $[\alpha-\ee +i y^{1-\alpha+\ee}, \alpha-\ee +iT]$,
\item $\cC_7$ est le segment $[\alpha-\ee+iT, \alpha+iT]$
. \end{enumerate}

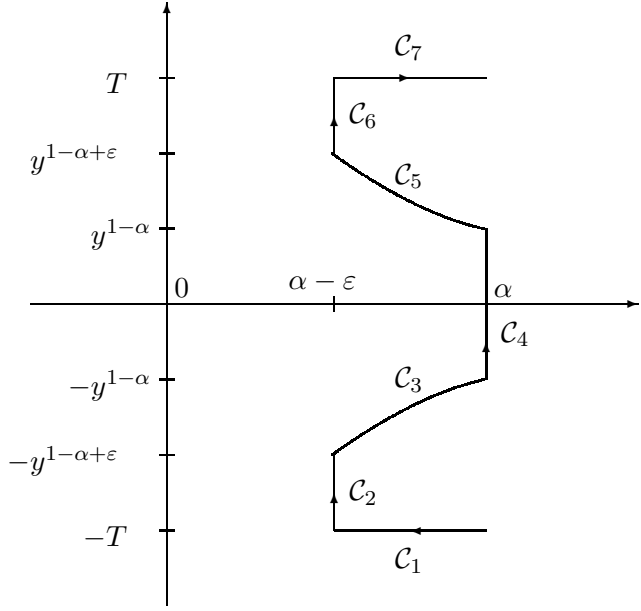
\begin{center}
\begin{picture}(80,80)
\linethickness{0.2mm}
\put(0, 40){ \vector(1, 0){80} }
\put(18, 0){ \vector(0, 1){80} }
\put(40, 39){ \line(0, 1){2} }
\put(17, 10){ \line(1, 0){2} }
\put(17, 20){ \line(1, 0){2} }
\put(17, 30){ \line(1, 0){2} }
\put(17, 50){ \line(1, 0){2} }
\put(17, 60){ \line(1, 0){2} }
\put(17, 70){ \line(1, 0){2} }

\put(60,  10){ \line(-1, 0){20} }
\put(40,  10){ \line(0, 1){10} }
\qbezier(41, 20)(52, 28)(61, 30)
\put(60, 30){ \line(0, 1){20} }
\qbezier(61, 50)(52, 52)(41, 60)
\put(40, 60){ \line(0, 1){10} }
\put(40, 70){ \line(1, 0){20} }

\put(60, 10){ \vector(-1, 0){ 10} }
\put(40, 10){ \vector(0, 1){ 5} }
\put(60, 30){ \vector(0, 1){ 5} }
\put(40, 60){ \vector(0, 1){ 5} }
\put(40, 70){ \vector(1, 0){ 10} }

\put( 10, 68){ $T$ }
\put( 0, 58){ $y^{1-\alpha+\ee}$ }
\put( 8, 48){ $y^{1-\alpha}$ }
\put( 5, 28){ $-y^{1-\alpha}$ }
\put( -3,  18){ $-y^{1-\alpha+\ee}$ }
\put( 7,  8){ $-T$ }
\put(19, 41){ $0$ }
\put(34, 42){ $\alpha-\ee$ }
\put(61, 41){ $\alpha$ }

\put(48, 5){ $\cC_1$ }
\put(42, 14){ $\cC_2$ }
\put(48, 29){ $\cC_3$ }
\put(62, 35){ $\cC_4$ }
\put(48, 56){ $\cC_5$ }
\put(42, 64){ $\cC_6$ }
\put(48, 73){ $\cC_7$ }

\end{picture}
\end{center}
Pour $j\in\{1, \cdots, 7\}$ on note $I_j$ la contribution du chemin $\cC_j$ :
\[ I_j = I_j(\eta) := \frac{1}{2i\pi} \int_{\cC_j} \zeta(s, q ; y) x^s \cPhi(\eta x, s) \dd s .\]

La contribution du segment $\cC_7$ est, grâce aux Lemmes~\ref{majo_l_decal} et~\ref{majo_phi},
\[ I_7 \ll \int_{\alpha-\ee}^{\alpha} \frac{2^{\omega(q)}q^{1-\sigma}}{\vphi(q)} \zeta(\alpha, y) x^{\sigma} \frac{1+|\eta|x}{T} \dd \sigma
\ll \frac{2^{\omega(q)}q^{1-\alpha}\Psi(x, y)}{\vphi(q)} \frac{(1+|\eta|x)\log x}{T} .\]

Sur le segment $\cC_6$ on a $|\tau| \geq y^{1-\sigma}$. Le Lemme~\ref{majo_l_decal} est donc encore applicable et on a
\begin{align*}
I_6 &\ \ll \frac{2^{\omega(q)}q^{1-\alpha+\ee}}{\vphi(q)} x^{\alpha-\ee/2} \zeta(\alpha, y) (1+|\eta|x) \log T 
\ll \frac{2^{\omega(q)}q^{1-\alpha} \Psi(x, y)}{\vphi(q)}\frac{(1+|\eta|x)\log T \log x}{x^{\ee/4}}
.\end{align*}

Sur le segment $\cC_5$, le traitement est analogue. On a
\begin{align*}
 I_5 &\ \ll \frac{2^{\omega(q)}}{\vphi(q)}
\int_{\alpha-\ee}^{\alpha}q^{1-\sigma} |\zeta(\sigma+iy^{1-\sigma}, y) \cPhi(\eta x, \sigma+iy^{1-\sigma})|
(\log y) x^{\sigma} y^{1-\sigma} \dd \sigma
.\end{align*}
Pour tout~$\kappa \in [0, \ee]$, on a pour un certain~$\delta>0$
\[ \zeta(\alpha-\kappa+iy^{1-\alpha+\kappa}, y) \ll \zeta(\alpha, y) H(u)^{-\delta} \]
\[ \cPhi(\eta x, \alpha-\kappa+iy^{1-\alpha+\kappa}) \ll y^{1-\alpha+\kappa}\log(2+|\eta|x)/(1+|\eta|x)^{\alpha-\kappa} \]
où la première inégalité est conséquence du Lemme~\ref{majo_l_decal} et de~\cite[lemma~8]{TeneHild86}. Ainsi on obtient
\begin{align*}
I_5 &\ \ll \frac{2^{\omega(q)}q^{1-\alpha} x^\alpha \zeta(\alpha, y)  (\log x) (\log y) y^{2(1-\alpha)}}{\vphi(q) (1+|\eta|x)^\alpha H(u)^\delta}
\int_0^{\ee} \left(\frac{y^2q(1+|\eta|x)}{x}\right)^\kappa \dd \kappa \\
&\ \ll \frac{2^{\omega(q)}q^{1-\alpha} \Psi(x, y)}{(1+|\eta| x)^\alpha \vphi(q) H(u)^{\delta/2}}
\end{align*}
où l'on a utilisé l'inégalité~$(\log x)(\log y) y^{2(1-\alpha)} \ll H(u)^{\delta/2}$ qui découle de nos hypothèses sur~$(x, y)$.

Sur le segment $\cC_4$, le traitement est identique à celui des segments $\cC_3$, $\cC_4$ et $\cC_5$
dans la preuve de la proposition~3.5 de~\cite{D2012}. On reprend ici les étapes principales. La contribution des~$s$
vérifiant~$1/\log y \leq |\tau| \leq y^{1-\alpha}$ est
\[ \ll \frac{2^{\omega(q)}q^{1-\alpha}x^\alpha}{\vphi(q)} \int_{1/\log y}^{y^{1-\alpha}} |\zeta(\alpha + i\tau) \cPhi(\eta x, \alpha+i\tau)| \dd \tau
\ll \frac{2^{\omega(q)} q^{1-\alpha} \Psi(x, y)}{\vphi(q) (1+|\eta|x)^\alpha H(u)^{\delta/2}} \]
où l'on a utilisé les Lemmes~\ref{majo_phi} et~\cite[lemma~8]{TeneHild86}.
On pose~$T_0 := u^{-1/3}(\log y)^{-1}$. La contribution à~$I_4$ des~$s$ vérifiant~$T_0 \leq |\tau| \leq 1/\log y$ est
\[ \ll \frac{2^{\omega(q)} q^{1-\alpha} \log(2+|\eta|x)}{\vphi(q) (1+|\eta|x)^\alpha}
\int_{T_0}^{1/\log y}|\zeta(\alpha+i\tau)|x^\alpha \dd \tau
\ll \frac{2^{\omega(q)}q^{1-\alpha} \Psi(x, y) \log(2+|\eta|x)}{\vphi(q) (1+|\eta|x)^\alpha u}\]
où l'on a utilisé les calculs de la démontration du Lemma~11 de~\cite{TeneHild86} pour évaluer l'intégrale.
Lorsque~$|\tau| \leq T_0$, et quitte à changer la valeur de~$c_1$ afin d'avoir~$\alpha \geq 1/2$, on a pour~$s = \alpha+i\tau$ l'estimation
\[ \int_0^1 \e(\lambda t) (\log t)^k t^{s-1} \dd t \ll_k \frac{\log(2+|\lambda|)^{k+1}}{(1+|\lambda|)^\alpha} \qquad (k \in \{0,1,2\}.)\]
Cela se montre par en intégrant par partie de façon similaire aux calculs du Lemme~\ref{majo_phi}. Ainsi,
\begin{align*}
\frac{s q^{1-s}}{\vphi(q)} \prod_{p|q}{\left( 1-p^{s-1} \right)}\cPhi(\eta x, s)
= &\ \frac{\alpha q^{1-\alpha}}{\vphi(q)} \prod_{p|q}{\left( 1-p^{\alpha-1} \right)}\cPhi(\eta x, \alpha) \\
&\ + \lambda\tau + O\left(\frac{2^{\omega(q)}q^{1-\alpha} (\log q)^2 (\log(2+|\eta|x))^3}{\vphi(q) (1+|\eta|x)^\alpha} \tau^2\right)
\end{align*}
où le coefficient~$\lambda$ dépend au plus de~$x, y, q$ et~$\eta$ et vérifie
\[\lambda \ll \frac{2^{\omega(q)} q^{1-\alpha} \log q (\log(2+|\eta|x))^2}{\vphi(q) (1+|\eta|x)^\alpha}.\]
Le reste des calculs sont identiques à ceux de~\cite[proposition~3.5]{D2012} : en reportant ce développement dans l'intégrale
puis en développant le terme complémentaire~$x^s \zeta(s, y) / s$ de la même façon que dans~\cite[lemma~11]{TeneHild86}, on obtient finalement
\begin{align*}
I_4 = \frac{\alpha q^{1-\alpha}}{\vphi(q)} \prod_{p|q}{\left( 1-p^{\alpha-1} \right)}\cPhi(\eta x, \alpha) \Psi(x, y)
+ O\left(\frac{2^{\omega(q)} q^{1-\alpha} \Psi(x, y)}{\vphi(q) (1+|\eta|x)^\alpha} \frac{(\log q)^2 (\log(2+|\eta|x))^3}{u}\right)
. \end{align*}

Pour $j\in\{1, 2, 3\}$ on a $I_j(\eta) = \overline{I_{8-j}(-\eta)}$ et on se ramène aux calculs précédents.
L'estimation voulue suit en regroupant toutes les contributions puisque l'on a toujours~$\ee \log x \gg \log T$.

\end{proof}

\subsection{Caractères principaux par la transformée de Laplace}

Une autre façon d'évaluer~$V(x, y ; q, \eta)$ consiste à utiliser une estimation de De Bruijn~\cite{deBruijn1951}
précisée par Saias~\cite{Saias1989} et utilisée dans~\cite{AGRB2010}. On rappelle la définition~\eqref{def_Ye}.

\begin{prop}\label{caracts_pcp_saias}
Pour tout~$\ee>0$ fixé, lorsque~$(x, y)\in (H_\ee)$ avec $y\leq \sqrt{x}$, et~$q \in \bfN$ et~$\eta\in\bfR$
avec~$q\leq \cY_{\ee}$ et~$|\eta| \leq \cY_{\ee}/x$, on a
\[ V(x, y ; q, \eta) = \Vt(x, y ; q, \eta) + O_\ee\left(\frac{2^{\omega(q)}\Psi(x, y)}{\vphi(q) \cY_{\ee}} \right) .\]
\end{prop}

\begin{proof}

Cette proposition généralise des calculs faits dans~\cite[théorème~4.2]{AGRB2010}.
La différence vient du fait que l'on calcule uniquement la contribution des caractères principaux,
pour lesquels on dispose de la région sans zéro de~$\zeta$ de Vinogradov-Korobov,
plus étendue que la région de Siegel-Walfisz pour les fonctions~$L$. Notons~$Q := x/\cY_{\ee}$.
Les mêmes calculs que~\cite[lemme 3.2]{AGRB2010} montrent que
\begin{align*}
V(x, y ; q, \eta) &\ = \sum_{k|q} \frac{\mu(q/k) k}{\vphi(q)} \sum_{m \in S(x/k, y)} \e(mk\eta) \\
&\ = \sum_{k|q} \frac{\mu(q/k)k}{\vphi(q)} \int_{Q}^{x} \e(t\eta) \dd\{\Psi(t/k, y)\}
+ O_\ee\left(\frac{2^{\omega(q)}\Psi(x, y)}{\vphi(q)\cY_{2\ee}}\right)
\end{align*}
où l'on a utilisé l'estimation~\eqref{estim_psi_local_crude} et le fait que~$\cY_{\ee}^{2\alpha-1} \gg_\ee \cY_{2\ee}$ sous notre hypothèse sur~$(x, y)$.
L'estimation~\eqref{estim_Psi_saias} fournit, en intégrant par parties,
\begin{align*}
\int_{Q}^x \e(t\eta) & \dd\{\Psi(t/k, y)\} - \int_{Q}^{x} \e(t\eta) \dd\{\Lambda(t/k, y)\} \\ &\
= \int_{Q}^x \e(t\eta)\dd\{O(\Psi(t/k, y)\cY_{\ee/2}^{-1}\} \\
&\ \ll
(1+|\eta|x)\frac{\Psi(x/k, y)}{\cY_{\ee/2}} \ll_\ee \frac{\Psi(x/k,y)}{\cY_{\ee}}
.\end{align*}
Notant~$V_q(x, y) := \sum_{k|q} \mu(q/k) k/ \vphi(q) \Lambda(x/k, y)$, on en déduit
\[ V(x, y ; q, \eta) = \int_{Q}^x \e(t \eta) \dd V_q(t, y) +
O_\ee\left(\frac{2^{\omega(q)}\Psi(x, y)}{\vphi(q) \cY_{2\ee}}\right) .\]
En utilisant~\eqref{Lambda_lambda}, on réécrit cela sous la forme
\[ V(x, y ; q, \eta) = \sum_{k|q} \frac{\mu(q/k)k}{\vphi(q)} \left\{ \int_{Q/k}^{x/k} \e(kt\eta) \lambda(t, y) \dd t +
\int_{Q/k}^{x/k} \e(k t \eta) t \dd\Big(\frac{\{t\}}{t}\Big) \right\}+ O\left(\frac{2^{\omega(q)}\Psi(x, y)}{\vphi(q) \cY_{2\ee}} \right) .\]
En intégrant par parties, on a d'une part
\[ \sum_{k|q} \frac{\mu(q/k)k}{\vphi(q)} \int_{Q/k}^{x/k} \e(k t \eta) t \dd\Big(\frac{\{t\}}{t}\Big)
\ll \frac{2^{\omega(q)} q \cY_\ee}{\vphi(q)} \ll \frac{2^{\omega(q)}\Psi(x, y)}{\vphi(q) \cY_\ee}, \]
et d'autre part, en utilisant~$\int_0^t \e(v\eta) \dd v = E(t, t ; \eta) + O(1)$,
\begin{align*}
\int_{Q/k}^{x/k} \e(k t \eta) \lambda(t, y) \dd t =  E\Big(\frac{x}{k}, \frac{x}{k} ; k\eta\Big) \lambda\Big(\frac{x}{k}, y\Big)
- \int_{Q/k}^{x/k} E(t, t ; k\eta) \lambda'(t, y) \dd t + O\left( y u + \Psi(Q/k, y) \right).
\end{align*}
Les hypothèses faites sur~$x$ et~$y$ assurent que~$y u \cY_0 \ll \Psi(x, y)$ et~$\cY_\ee^{2\alpha-1} \gg \cY_{2\ee}$, le terme d'erreur
est donc~$O(\Psi(x, y) / (k \cY_{2\ee}))$. On a de plus la majoration
\begin{equation}\label{majo_lambda'}
\lambda'(t, y) \ll \frac{\Psi(t, y) \log(u+1) }{t^2 \log y}  \qquad (y \leq t \leq x)
.\end{equation}
qui se déduit par différentiation de~\cite[formule~(2.2)]{AGRB2010} en utilisant par exemple l'estimation~\cite[formule~(30)]{rdlb99}.
Cela implique
\[ \int_{1}^{Q/k} E(t, t ; k\eta) \lambda'(t, y) \dd t \ll \Psi(x/\cY_{\ee/4}, y) \log x + \frac{\Psi(x/(k\cY_\ee), y)}{x/\cY_{\ee/4}} \]
en séparant l'intégrale en~$x/\cY_{\ee/4}$. Les hypothèses sur~$x$ et~$y$ impliquent alors
que chacun de ces termes est~$\ll \Psi(x, y)/(k \cY_{2\ee})$. On a enfin
\begin{equation}\label{Vt_ipp}
\begin{aligned}
&\ E(x/k, x/k ; k\eta) \lambda(x/k, y) - \int_{1}^{x/k} E(t, t ; k\eta) \lambda'(t, y) \dd t \\
= &\ \sum_{n\leq x/k} \e(kn\eta) \lambda(n, y) + \frac{y}{\log y} \sum_{y < n \leq x/k} \left( \frac{\{x/(ky)\}}{x/k} - \frac{\{n/y\}}{n} \right) \\
= &\ \sum_{n\leq x/k} \e(kn\eta) \lambda(n, y) + O(y u)
.\end{aligned}
\end{equation}
De même que précédemment, le terme d'erreur est~$O(\Psi(x, y)/(k\cY_{2\ee}))$. On obtient donc
\begin{align*}
V(x,y ; q, \eta) &\ = \sum_{k|q} \frac{\mu(q/k)k}{\vphi(q)} \sum_{n\leq x/k} \e(kn\eta) \lambda(n, y)
+ O\left(\frac{2^{\omega(q)}\Psi(x, y)}{\vphi(q) \cY_{2\ee}} \right) \\
&\ = \Vt(x, y ; q, \eta) + O\left(\frac{2^{\omega(q)}\Psi(x, y)}{\vphi(q) \cY_{2\ee}} \right) \\
\end{align*}
qui est l'estimation voulue, $\ee$ pouvant être pris arbitrairement petit.

\end{proof}

\subsection{Caractères exceptionnels}

Soit~$Q$ un entier supérieur à~$2$. On se place dans le cas de l'existence du zéro de Siegel.
Soit~$q\leq Q$ avec $\DSZ(q) = 1$ et~$P(q)\leq y$. On définit
\begin{align*}
LW(s, q ; y) := &\ \frac{\tau(\chi_1)}{\vphi(q_1)} \sum_{r|(q/q_1)}{ \frac{\mu(r)\chi_1(r)}{\vphi(r)}
\left(\frac{q_1 r}{q}\right)^s L(s, \chi_r ; y) } \\
= &\ \left(\frac{q}{q_1}\right)^{1-s} \frac{\tau(\chi_1)}{\vphi(q)}
\prod_{\substack{p | q/q_1 \\ p \nmid q_1}}{\left(1-\chi_1(p) p^{s-1}\right)}  L(s, \chi_1 ; y)
.\end{align*}
Pour~$\sigma \leq 1$, le facteur devant~$L(s, \chi_1 ; y)$ est $\ll (q/q_1)^{1-\sigma}2^{\omega(q/q_1)}\sqrt{q_1}/\vphi(q)$.

\begin{prop}\label{caracts_sieg}
Il existe des constantes~$c_1$ et~$c_2$ positives telles que pour tout~$(x, y)$
avec~$(\log x)^{c_1} \leq y \leq \exp\{(\log x)/(\log\log x)^4\}$,
tout~$Q \leq y^{c_2/\log \log \log x}$, tout caractère~$\chi$ de module~$q\leq Q$, qui est~$Q$-exceptionnel, et tout~$\eta\in\bfR$
vérifiant~$|\eta|x \leq x^{1/4}$, la quantité~$W(x, y ; q, \eta)$ soit un grand~$O$ de
\begin{equation}\label{majo_W}
\begin{aligned}
\frac{2^{\omega(q/q_1)}\sqrt{q_1}\Psi(x, y)}{\vphi(q)} \left( \frac{q}{q_1}\right)^{1-\alpha}
\Bigg( \frac{1}{(1+|\eta|x)^{\alpha+\beta-1} x^{1-\beta} H(u)^{\delta} } + (1+|\eta|x) R(x, y, Q)  \Bigg)
\end{aligned}
\end{equation}
où~$R(x, y, Q) = y^{-c_2/\log\log\log x} + \cL^{-c_2} + (\log x)^{3/2} x^{-c_2/\log Q}$.
\end{prop}
\begin{remarque}
La contrainte sur~$Q$ n'est pas limitante dans les applications que l'on envisage.
L'approche adoptée dans~\cite[Lemma~5.2]{soundararajan2008distribution},
permet d'obtenir une majoration moins forte mais qui est valable lorsque~$Q$ est de l'ordre d'une petite puissance de~$y$.
Cela n'est pas étudié ici.
\end{remarque}

\begin{proof}
On pose~$T := \min\{y^{c_2/\log \log \log x}, \cL\}$ et~$\ee = c_3/\log QT$, $c_2$ étant choisie suffisamment petite
pour que les hypothèses du Lemme~\ref{estim_lsieg} vis-à-vis de~$T$ et~$Q$ soient vérifiées, et~$c_3$ étant choisie
plus petite que la constante~$c_2$ du Lemme~\ref{estim_lsieg}. Lorsque~$\beta \leq 1-\sqrt{c_2}/\log QT$, la Proposition~\ref{caracts_norm}
s'applique et on obtient, de la même façon qu'à la Proposition~\ref{prop_E_VW},
\begin{align*}
W(x, y ; q, \eta) &\ = \frac{\tau(\chi_1)}{\vphi(q_1)} \sum_{r|(q/q_1)}
\frac{\mu\left(\frac{q}{q_1r}\right)\chi_1\left(\frac{q}{q_1r}\right)}{\vphi\left(\frac{q}{q_1r}\right)} \Psi_0(x/r, y ; \chi_{q/(q_1r)}, r\eta) \\
&\ \ll \Psi(x, y) (1+|\eta|x) \left( y^{-c_2} + \cL^{c_2} + (\log x)^{3/2} x^{-c_2/\log Q} \right)
\frac{\sqrt{q_1}}{\vphi(q_1)}\left(\frac{q}{q_1}\right)^{-\alpha} \sum_{r|(q/q_1)} \frac{\mu^2(r)}{\vphi(r)} \\
&\ \ll \frac{2^{\omega(q/q_1)}\sqrt{q_1}}{\vphi(q)} \left(\frac{q}{q_1}\right)^{1-\alpha} \Psi(x, y) (1+|\eta|x) R
\end{align*}
qui est de l'ordre de la majoration annoncée. On suppose maintenant~$1-\beta \leq \sqrt{c_2}/\log QT$.
De même qu'à la Proposition~\ref{caracts_pcp}, par les Lemmes~\ref{hypoth_perron_a} et~\ref{perron_phi}, on a
\begin{equation*}
\begin{aligned}
 W(x, y ; q, \eta) = &\ \frac{1}{2i\pi}\int_{\alpha-iT}^{\alpha+iT}{LW(s, q ; y) x^s \cPhi(\eta x, s) \dd s} \\
&\ + O\left(\frac{2^{\omega(q/q_1)}\sqrt{q_1}}{\vphi(q)}\left(\frac{q}{q_1}\right)^{1-\alpha} \Psi(x, y) \frac{\log x \log T}{T^{\alpha/2}} \right)
\end{aligned}
\end{equation*}
On déforme le contour pour suivre le chemin~$\cup_{i=1}^{7}{\cC'_i}$, où
\begin{enumerate}
\item $\cC'_1$ est le segment $[\alpha-iT, \alpha+\beta-1-\ee -iT]$,
\item $\cC'_2$ est le segment $[\alpha+\beta-1-\ee -iT, \alpha+\beta-1-\ee -i y^{1-\alpha+\ee}]$,
\item $\cC'_3$ est le chemin reliant le point $\alpha+\beta-1-\ee-i y^{1-\alpha+\ee}$
au point $\alpha+\beta-1 - i y^{1-\alpha}$ suivant la courbe $\tau = y^{\beta-\sigma}$,
\item $\cC'_4$ est le segment $[\alpha+\beta-1 - i y^{1-\alpha}, \alpha+\beta-1 + i y^{1-\alpha}]$,
\item $\cC'_5$ est le chemin reliant le point $\alpha+\beta-1 + i y^{1-\alpha}$
au point $\alpha+\beta-1-\ee +i y^{1-\alpha+\ee}$ suivant la courbe $\tau = -y^{\beta-\sigma}$,
\item $\cC'_6$ est le segment $[\alpha+\beta-1-\ee +i y^{1-\alpha+\ee}, \alpha+\beta-1-\ee +iT]$,
\item $\cC'_7$ est le segment $[\alpha+\beta-1-\ee+iT, \alpha+iT ]$.
\end{enumerate}
Pour $j\in\{1, \cdots, 7\}$, on note $I'_j$ la contribution du chemin $\cC'_j$ :
\[ I'_j = I'_j(\eta) := \frac{1}{2i\pi} \int_{\cC'_j} LW(s, q ; y) x^s \cPhi(\eta x, s) \dd s .\]

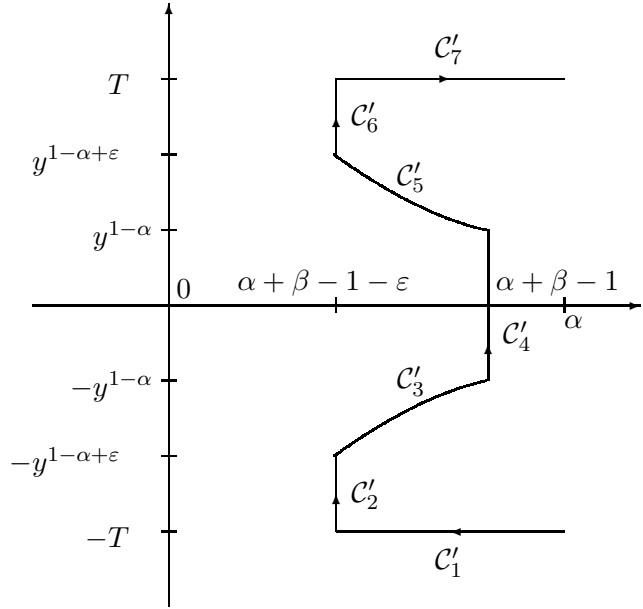
\begin{center}
\begin{picture}(80,80)
\linethickness{0.2mm}
\put(0, 40){ \vector(1, 0){80} }
\put(18, 0){ \vector(0, 1){80} }
\put(40, 39){ \line(0, 1){2} }
\put(17, 10){ \line(1, 0){2} }
\put(17, 20){ \line(1, 0){2} }
\put(17, 30){ \line(1, 0){2} }
\put(17, 50){ \line(1, 0){2} }
\put(17, 60){ \line(1, 0){2} }
\put(17, 70){ \line(1, 0){2} }
\put(70, 39){ \line(0, 1){2} }

\put(70,  10){ \line(-1, 0){30} }
\put(40,  10){ \line(0, 1){10} }
\qbezier(41, 20)(52, 28)(61, 30)
\put(60, 30){ \line(0, 1){20} }
\qbezier(61, 50)(52, 52)(41, 60)
\put(40, 60){ \line(0, 1){10} }
\put(40, 70){ \line(1, 0){30} }

\put(70, 10){ \vector(-1, 0){ 15} }
\put(40, 10){ \vector(0, 1){ 5} }
\put(60, 30){ \vector(0, 1){ 5} }
\put(40, 60){ \vector(0, 1){ 5} }
\put(40, 70){ \vector(1, 0){ 15} }

\put( 10, 68){ $T$ }
\put( 0, 58){ $y^{1-\alpha+\ee}$ }
\put( 8, 48){ $y^{1-\alpha}$ }
\put( 5, 28){ $-y^{1-\alpha}$ }
\put(-3,  18){ $-y^{1-\alpha+\ee}$ }
\put( 7,  8){ $-T$ }
\put(19, 41){ $0$ }
\put(27, 42){ $\alpha+\beta-1-\ee$ }
\put(61, 42){ $\alpha+\beta-1$ }
\put(70, 37){ $\alpha$ }

\put(53, 5){ $\cC'_1$ }
\put(42, 14){ $\cC'_2$ }
\put(48, 29){ $\cC'_3$ }
\put(62, 35){ $\cC'_4$ }
\put(48, 56){ $\cC'_5$ }
\put(42, 64){ $\cC'_6$ }
\put(53, 73){ $\cC'_7$ }

\end{picture}
\end{center}

\bigskip

De la même façon que dans la démonstration de la Proposition~\ref{caracts_pcp}, on a
\begin{align*}
I'_7 \ll &\ \int_{\alpha+\beta-1-\ee}^{\alpha} \frac{2^{\omega(q/q_1)}\sqrt{q_1}}{\vphi(q)}\left(\frac{q}{q_1}\right)^{1-\sigma}
\frac{\zeta(\alpha, y) x^{(\alpha-\sigma)/2} x^\sigma (1+|\eta|x)}{T} \dd \sigma \\
 \ll &\ \frac{2^{\omega(q/q_1)}\sqrt{q_1}}{\vphi(q)} \left(\frac{q}{q_1}\right)^{1-\alpha} \frac{\Psi(x, y)(1+|\eta|x) \log x}{T}
.\end{align*}

Sur le segment $\cC'_6$, on a encore~$|\tau| \geq y^{\beta-\sigma}$, d'où par le Lemme~\ref{majo_l_decal},
\begin{align*}
I'_6 \ll &\ \int_{y^{1-\alpha+\ee}}^{T} \frac{2^{\omega(q/q_1)}\sqrt{q_1}}{\vphi(q)} \left(\frac{q}{q_1}\right)^{2-\alpha-\beta+\ee}
\frac{\zeta(\alpha, y) x^{(\beta-1)/2-\ee/2} x^\alpha (1+|\eta|x)}{\tau} \dd \tau \\
\ll &\ \frac{2^{\omega(q/q_1)}\sqrt{q_1}}{\vphi(q)}\left(\frac{q}{q_1}\right)^{1-\alpha} \frac{\Psi(x, y)(1+|\eta|x)\log x \log T}{x^{\ee/4}}
\end{align*}
quitte à supposer~$c_2$ petite, afin d'avoir~$q/q_1 \leq x^{1/4}$.

Sur le chemin $\cC'_5$, les Lemme~\ref{majo_l_decal} et~\ref{estim_lsieg} permettent d'écrire pour un certain $\delta>0$,
\begin{align*}
|L(\sigma+i\tau, \chi_1 ; y)| \ll &\ |L(\alpha+\beta-1+i\tau, \chi_1 ; y)|x^{(\alpha+\beta-1-\sigma)/2} \\
\ll &\ \zeta(\alpha, y) x^{(\alpha+\beta-1-\sigma)/2} H(u)^{-\delta}
.\end{align*}
En remarquant que~$\ee \log q \ll 1$, on obtient
\begin{align*}
I'_5 \ll &\ \int_{\alpha+\beta-1-\ee}^{\alpha+\beta-1} \frac{2^{\omega(q/q_1)}\sqrt{q_1}}{\vphi(q)} \left(\frac{q}{q_1}\right)^{1-\sigma}
\zeta(\alpha, y) x^{(\alpha+\beta-1+\sigma)/2} H(u)^{-\delta} 
\frac{(\log y)y^{2(\beta-\sigma)}\log(2+|\eta|x)}{(1+|\eta|x)^{\sigma}} \dd \sigma \\
 \ll&\ \frac{2^{\omega(q/q_1)}\sqrt{q_1}}{\vphi(q)} \left(\frac{q}{q_1}\right)^{1-\alpha}
\frac{\Psi(x, y)}{(1+|\eta|x)^{\alpha+\beta-1} x^{1-\beta} H(u)^{\delta/2}}
\end{align*}
où l'on a utilisé l'hypothèse~$\log y \leq (\log x) / (\log \log x)^4$ sous la
forme~$(\log x) H(u)^{-\eta} \ll_\eta 1$ pour tout~$\eta>0$.

Sur $\cC'_4$ les calculs sont similaires : par le Lemme~\ref{estim_lsieg} on a
\begin{align*}
I'_4 \ll&\ \int_{0}^{y^{1-\alpha}} \frac{2^{\omega(q/q_1)}\sqrt{q_1}}{\vphi(q)} \left(\frac{q}{q_1}\right)^{2-\alpha-\beta}
\zeta(\alpha, y) H(u)^{-\delta} x^{\alpha+\beta-1} \frac{(1+|\tau|)\log(2+|\eta|x)}{(1+|\eta|x)^{\alpha+\beta-1}} \dd \tau \\
\ll &\ \frac{2^{\omega(q/q_1)\sqrt{q_1}}}{\vphi(q)} \left(\frac{q}{q_1}\right)^{1-\alpha}
\frac{\Psi(x, y)}{(1+|\eta|x)^{\alpha+\beta-1} x^{1-\beta} H(u)^{\delta/2}}
\end{align*}
ce qui est de l'ordre de grandeur souhaité.

Pour~$j\in\{1, 2, 3\}$, le caractère $\chi_1$ étant réel, la même remarque qu'à la démonstration de la Proposition~\ref{caracts_pcp}
est valable : on a $I'_j(\eta) = \overline{I'_{8-j}(-\eta)}$ et les majorations qui concernent $I'_{8-j}$ s'appliquent.

En regroupant les différentes contributions et en observant que
\[ \frac{1}{T^{\alpha/2}} + \frac{1}{T} + \frac{\log x \log T}{x^{\ee/8}} \ll R \]
quitte à réduire la valeur de~$c_2$, on obtient la majoration annoncée.

\end{proof}

\subsection{Démonstration du Théorème~\ref{thm_E_pcp}}

On pose~$Q = \lceil T_2^{c_2} \rceil$, en observant que~$\log x / \log Q \gg \log \cL$ lorsque~$c_2 \leq 1$.
Quitte à supposer~$c_1$ suffisamment grande et~$c_2$ suffisamment petite, $x$, $y$ et~$Q$
vérifient les hypothèses des Propositions~\ref{prop_E_VW}, \ref{caracts_pcp}, et~\ref{caracts_sieg}.
Le terme d'erreur provenant de l'estimation~\eqref{estim_E_VW} est
\[ \ll (1+|\eta|x) \Psi(x, y) \left( y^{-c_3/\log\log\log x} + \cL^{-c_3} \right) \ll \Psi(x, y) T_2^{-c_2} \]
pour une certaine constante~$c_3$. Le terme d'erreur provenant de l'estimation~\eqref{estim_V_col} est
\begin{align*}
& \ll \frac{2^{\omega(q)}q^{1-\alpha}\Psi(x, y)}{\vphi(q)} \left( \frac{(\log q)^2 (\log(2+|\eta|x))^3}{(1+|\eta|x)^\alpha u}
+ (1+|\eta|x) \left( y^{-c_3} + \e^{-c_3 (\log x)^{3/5} (\log \log x)^{-1/5} } \right) \right) \\
& \ll \frac{2^{\omega(q)}q^{1-\alpha}\Psi(x, y)}{\vphi(q) (1+|\eta|x)^\alpha} \frac{(\log q)^2 (\log(2+|\eta|x))^3}{u}
+ \frac{\Psi(x, y)}{T_1^{c_2}} 
\end{align*}
Enfin, la quantité~$R$ intervenant dans~\eqref{majo_W} est~$\ll T_2^{-c_2}$. Ceci implique l'estimation~\eqref{estim_E_pcp}.

Si on suppose de plus que~$(x, y)\in (H_\ee)$, $q\leq \cY_\ee$ et~$|\eta|x \leq \cY_\ee/q$ pour un certain~$\ee>0$,
alors en évaluant~$V(x, y ; q, \eta)$ par la Proposition~\ref{caracts_pcp_saias} plutôt que~\ref{caracts_pcp},
on obtient l'estimation~\eqref{estim_E_BG}.

\section{En norme $L^2$}

On s'intéresse ici à l'obtention d'une majoration pour le deuxième moment de~$V(x, y ; q, \eta)$ et~$W(x, y ; q, \eta)$.
Lorsque~$2\leq y \leq x$, on a
\[ \int_0^1{|E(x, y ; \vth)|^2 \dd \vth} = \Psi(x, y) .\]
Le lemme qui suit est une majoration de même ordre de grandeur pour les normes~$L^2$
sur les arcs majeurs de~$V(x, y ; q, \eta)$ et~$W(x, y ; q, \eta)$,
qui sont les termes principaux apparaissant dans l'estimation~\eqref{estim_E_VW}.

\begin{prop}\label{prop_VW_L2}

Lorsque~$2\leq y \leq x$, $Q\geq 2$ et~$R\leq x$, on a
\[ \sum_{q \leq R}{ \vphi(q) \int_{-1/(qQ)}^{1/(qQ)}{ |V(x, y ; q, \eta)|^2 \dd \eta} } \ll R^{1-\alpha} (\log x)^5 \Psi(x,y) \]
\[ \sum_{q \leq R}{ \vphi(q) \DSZ(q) \int_{-1/(qQ)}^{1/(qQ)}{ |W(x, y ; q, \eta)|^2 \dd \eta} }
\ll \frac{q_1^2}{\vphi(q_1)^2} R^{1-\alpha} (\log x)^5 \Psi(x,y) \]

\end{prop}

On note que $q_1/\vphi(q_1) \ll \log \log q_1$.

\begin{proof}
Soit $q$ avec $\DSZ(q)=1$. On note pour simplifier $r_1 := q/q_1$. On a
\begin{align*}
W(x, y ; q, \eta) = &\ \frac{\tau(\chi_1)}{\vphi(q_1)} \sum_{ r | r_1 } {\frac{\mu(r)\chi_1(r)}{\vphi(r)} \sum_{m \in S(x r/r_1, y)}{ \e\left(\frac{m r_1}{r} \eta\right) \chi_r(m) } } \\
= &\ \frac{\tau(\chi_1)}{\vphi(q_1)} \sum_{r' | r_1}{ \sum_{m \in S(x/r', y)}{ \frac{\mu\left(\frac{r_1}{r'}\right)\chi_1\left(\frac{r_1}{r'}\right)\chi_{\frac{r_1}{r'}}(m)}{\vphi\left(\frac{r_1}{r'}\right)} \e(mr'\eta)}} \\
= &\ \frac{\tau(\chi_1)}{\vphi(q_1)} \sum_{n \in S(x, y)}{ \sum_{r' | (r_1, n)}{ \e(n\eta) \frac{\mu\left(\frac{r_1}{r'}\right)\chi_1\left(\frac{r_1}{r'}\right)\chi_{\frac{r_1}{r'}}\left(\frac{n}{r'}\right)}{\vphi\left(\frac{r_1}{r'}\right)} }}
.\end{align*}
Notons temporairement
$w_{r'}(n) := \mu(r_1/r')\chi_1(r_1/r')\chi_{r_1/r'}(n/r') / \vphi(r_1/r')$.
La présence du terme en~$\chi_{r_1/r'}$ annule $w_{r'}(n)$ sauf si $(n/r',
q/r') = 1$ soit $r' = (q, n)$. En particulier, lorsque~$r' < (r_1, n)$ on a~$w_{r'}(n)=0$ et on obtient
\[ W(x, y ; q, \eta) = \frac{\tau(\chi_1)}{\vphi(q_1)} \sum_{n \in S(x, y)}{ \e(n\eta) w_{(r_1,n)}(n) } .\]
On a donc
\begin{align*}
& I := \sum_{\substack{q \leq R \\ q_1|q}} { \vphi(q) \int_{-1/(qQ)}^{1/(qQ)}{ \left| W(x, y ; q, \eta) \right|^2 \dd \eta} } \\
 = &\ \frac{\tau(\chi_1)^2}{\vphi(q_1)^2} \sum_{\substack{q\leq R \\ q_1|q}}{ \vphi(q) \Bigg(
\sum_{n \in S(x, y)}{ \frac{2}{qQ} w_{(r_1,n)}(n)^2 } + 
\sum_{\substack{n,m \in S(x, y) \\ m \neq n}}{ \frac{\sin\left(\frac{2\pi (m-n)}{qQ}\right)}{\pi (m-n)} w_{(r_1,n)}(n) w_{(r_1,m)}(m) }
\Bigg) }
.\end{align*}
On a $\sin(2\pi(m-n)/(qQ)) / (\pi (m-n)) \ll 1/(qQ + |m-n|)$. La majoration~$w_{r'}(n) \ll 1/\vphi(r_1/r')$ fournit donc
\begin{align*}
I \ll \frac{q_1}{\vphi(q_1)^2} \sum_{\substack{q\leq R \\ q_1 | q}}  \vphi(q) \sum_{n \in S(x, y)} \frac{1}{\vphi\left(\frac{r_1}{(r_1,n)}\right)}
\Bigg( &\ \frac{1}{qQ}\sum_{n \leq m \leq n+qQ}{\frac{1}{\vphi\left(\frac{r_1}{(r_1,m)}\right)}}  + \sum_{n+qQ < m \leq x}{\frac{1}{(m-n)\vphi\left(\frac{r_1}{(r_1,m)}\right)}} \Bigg) 
.\end{align*}
Le premier terme dans la parenthèse intérieure est
\[ \leq \frac{1}{qQ} \sum_{d|r_1}{ \sum_{\substack{ \frac{n}{d} \leq m' \leq \frac{n+qQ}{d} }}{ \frac{1}{\vphi(r_1/d)} }} \leq \frac{\tau(r_1)}{\vphi(r_1)} .\]

Le second terme est
\begin{align*}
 \leq \sum_{d|r_1}{ \frac{1}{\vphi(r_1/d)} \sum_{\substack{\frac{n+qQ}{d}<m'\leq \frac{x}{d}}}{ \frac{1}{m'd-n} } }
 \leq \sum_{d|r_1}{ \frac{1}{\vphi(r_1/d)} \int_{\frac{n+qQ}{d}-1}^{\frac{x}{d}}{ \frac{1}{td-n} \dd t } } \ll (\log x) \frac{\tau(r_1)}{\vphi(r_1)} 
.\end{align*}
En utilisant~\eqref{estim_psi_local_crude}, on obtient donc, en utilisant~$\Psi(x/d, y) \ll r_1^{1-\alpha} \Psi(x, y)/d \quad(d\leq r_1)$,
\begin{align*}
I&\ \ll (\log x) \frac{q_1}{\vphi(q_1)^2} \sum_{\substack{q\leq R \\ q_1 | q}} {
\frac{\vphi(q)\tau(r_1)}{\vphi(r_1)} \sum_{d|r_1}{ \sum_{\substack{n \in S(x, y)\\d|n}}{ \frac{1}{\vphi(r_1/d)} } }} \\
&\ \ll (\log x) \frac{q_1}{\vphi(q_1)^2} \Psi(x, y) \sum_{r_1 \leq R/q_1} \frac{\vphi(q)\tau(r_1)^2 r_1^{1-\alpha}}{\vphi(r_1)^2}
.\end{align*}
On a $\vphi(q_1r_1) \leq q_1\vphi(r_1)$, la somme en~$r_1$ est donc $\ll R^{1-\alpha} q_1 (\log R)^4 \ll R^{1-\alpha} q_1 (\log x)^4$, et on obtient
\begin{align*}
I \ll \frac{q_1^2}{\vphi(q_1)^2}R^{1-\alpha} (\log x)^5 \Psi(x, y)
\end{align*}

On remarque que l'on n'a fait aucune hypothèse spécifique à $\chi_r$ et $q_1$ pour mener ce calcul.
Le cas~$q_1=1$, $\chi_r$ étant alors le caractère principal de module $r$, mène à la majoration
\begin{align*}
\sum_{q  \leq R} {\vphi(q) \int_{-1/(qQ)}^{1/(qQ)}{ \left| V(x, y ; q, \eta) \right|^2 \dd \eta} }
\ll R^{1-\alpha} (\log x)^5 \Psi(x, y)
\end{align*}

\end{proof}

\section{Application à un théorème de Daboussi}

\begin{proof}[Démonstration du Théorème~\ref{thm_daboussi}]
On suit la démonstration de~\cite[théorème~1.5]{tenerdlb2005turan}. Le lecteur peut s'y référer pour les détails.
On ne reprend ici que les étapes intermédiaires.
Soit~$c$ la constante donnée par le Théorème~\ref{thm_E_negl} et~$Y : [2, \infty[ \to \bfR$ une fonction croissante
telle que pour tout~$x\geq 2$, on ait~$(x, Y(x)) \in \cD_{c}$.
Soit~$\vth$ un irrationnel et~$f : \bfN \to \bfC$ une fonction multiplicative, on suppose
\[ \sum_{n\in S(x,y)} |f(n)|^2 \leq K_f \Psi(x,y) \quad (Y(x) \leq y \leq x)\]
pour un certain réel~$K_f>0$ dépendant au plus de~$f$.
On suppose~$Y(x) \leq y \leq x$ ; en particulier~$\alpha \geq 3/4$ pour~$x$ et~$y$ assez grands.
On note $E_f(x, y ; \vth) := \sum_{n\in S(x,y)} f(n)\e(n\vth)$.
Les calculs faits dans~\cite[formule (8.6)]{tenerdlb2005turan},
qui découlent d'une forme duale de l'inégalité de Turán-Kubilius~\cite[théorème 1.2]{tenerdlb2005turan},
montrent que pour tout~$z\geq 2$,
\[ E_f(x, y ; \vth) = \frac{1}{L(z)} \sum_{\substack{p\leq z \\ |f(p)|\leq 2}}
\sum_{\substack{m\in S(x/p,y)}} f(p) f(m) \e(mp\vth)  + O\left(\frac{\sqrt{K_f} \Psi(x,y)}{\sqrt{L(z)}}\right) \]
où l'on a noté
\[ L(z) := \sum_{\substack{p\leq z \\ |f(p)|\leq 2}} \frac{1-p^{-\alpha}}{p^\alpha} \asymp \sum_{\substack{p\leq z \\ |f(p)|\leq 2}} \frac{1}{p^\alpha} .\]
On a également, d'après~\cite[lemma 1]{daboussi1975fonctions}, pour un certain réel~$z_0 = z_0(f)\geq 2$ et tout~$z\geq z_0$,
\[ L(z) \gg \sum_{\substack{p\leq z \\ |f(p)|\leq 2}} \frac{1}{p} \gg \log \log z \]
grâce à l'hypothèse faite sur~$f$. Toujours d'après les calculs faits dans~\cite{tenerdlb2005turan}, par une inégalité de Cauchy-Schwarz, on a
\begin{equation}\label{majo_tp_Ef}
\begin{aligned}
&\ \sum_{\substack{p\leq z \\ |f(p)|\leq 2}} \sum_{\substack{m\in S(x/p,y)}} f(p) f(m) \e(mp\vth) \\
\ll &\ \sqrt{K_f \Psi(x, y)} \Bigg( \sum_{\substack{p\leq z\\|f(p)| \leq 2}} \Psi(x/p, y)
+ \sum_{\substack{p < q\leq z\\|f(p)|, |f(q)| \leq 2}} |E(x/p, y ; (p-q)\vth)| \Bigg)^{1/2}
.\end{aligned}
\end{equation}
Le Théorème~\ref{thm_E_negl} dans le cas~$y \leq \exp\{(\log x)/(\log \log x)^4\}$, et par exemple~\cite[théorème~1.5]{tenerdlb2005turan}
dans le cas contraire impliquent que chaque terme de la seconde somme du membre de droite est~$o_{\vth, p, q}(\Psi(x/p,y))$ lorsque~$x$
et~$y$ tendent vers l'infini en vérifiant~$Y(x) \leq y \leq x$. Le nombre de termes est borné par une fonction de~$z$,
par l'estimation~\eqref{estim_psi_local_crude}, le membre de gauche de~\eqref{majo_tp_Ef}
est donc~$\ll \sqrt{K_f} \Psi(x, y) (\sqrt{L(z)} + o_{\vth, z}(1))$ pour tout~$z\geq z_0$ fixé, ainsi
\[ \limsup_{\substack{x, y \to \infty \\ Y(x) \leq y \leq x}} { \frac{E_f(x, y ; \vth)}{\sqrt{K_f} \Psi(x, y)} }
\ll  \frac{1}{\sqrt{L(z)}} \]
et le résultat voulu suit en faisant tendre~$z$ vers l'infini.
\end{proof}

\section{Application aux sommes friables d'entiers friables}

On rappelle que~$N(x,y)$ a été défini en~\eqref{def_Nxy}.

\begin{proof}[Démonstration du Théorème~\ref{thm_abc}]
On a pour tous~$x$ et~$y$ avec~$2\leq y \leq x$,
\[ N(x, y) = \int_0^1 E(x, y ; \vth)^2 E(x, y ; -\vth) \dd \vth .\]
Soit $c>0$ et~$(x,y)\in\cD_c$. Lorsque~$y \geq \exp\{(\log x)/(\log \log x)^4\}$, le résultat de La Bretèche et Granville~\cite[théorème~1.1]{AGRB2010} est valable,
on suppose donc~$y \leq \exp\{(\log x)/(\log \log x)^4\}$. On note~$Q := \lceil x/\cL^{c_2}\rceil$, $R := \lceil \cL^{c_2} \rceil$,
et on suppose~$c$ suffisamment grande et~$c_2$ suffisamment petite pour que les hypothèses des Propositions~\ref{prop_E_VW}, \ref{caracts_pcp}
et~\ref{caracts_sieg} soient vérifiées pour~$x$, $y$ et~$R$ (dans le rôle de~$Q$).
Lorsque~$\vth$ vérifie~$q(\vth, Q) > R$, on a d'après le Lemme~\ref{majo_E_arcmin},
\[ E(x, y ; \vth) \ll  x \cL^{-c_3} \ll \Psi(x, y) \cL^{-c_3/2} \]
pour une certaine constante~$c_3>0$, quitte à supposer~$c> 2/c_3$.
La contribution des~$\vth$ vérifiant~$q(\vth, Q) > R$ est donc
\[ \ll \frac{\Psi(x, y)}{\cL^{c_3/2}} \int_0^1 |E(x, y ; \vth)|^2 \dd \vth = \frac{\Psi(x, y)^2}{\cL^{c_3/2}} \ll \frac{\Psi(x, y)^3}{x \cL^{c_3/4}} \]
quitte à supposer~$c> 4/c_3$. Lorsque~$q=q(\vth, Q) \leq R$, on écrit~$\vth=a/q+\eta$ avec~$|\eta| \leq 1/(qQ)$.
On remarque que~$q(-\vth, Q) = q(\vth, Q)$. La Proposition~\ref{caracts_norm} assure l'existence de~$c_4>0$ telle que
\[ E(x, y ; \vth) = V(x, y ; q, \eta) + \DSZ(q) \chi_1(a) W(x, y ; q, \eta) + O(\Psi(x, y) \cL^{-c_4}) .\]
Grâce à la Proposition~\ref{prop_VW_L2}, on a
\begin{align*}
&\ \frac{\Psi(x,y)}{\cL^{c_4}} \sum_{q\leq R} \sum_{\substack{a=1 \\ (a, q)=1}}^q
\int_{-1/(qQ)}^{1/(qQ)} \left|V(x, y ; q, \eta) + \DSZ(q) \chi_1(a) W(x, y ; q, \eta) \right|^2 \dd \eta \\
= &\ O\left( (\log \log x)^2(\log x)^3 \Psi(x, y)^2 \cL^{-c_4} \right) = O\left(\frac{\Psi(x, y)^3}{x \cL^{c_4/2}}\right)
\end{align*}
quitte à supposer~$c >2/c_4$. En notant que l'on a~$\chi_1(a)^2=1$ lorsque~$(a,q)=1$, et que~$\sum_{(a, q)=1} \chi_1(a) = 0$,
on obtient pour un certain réel~$c_5>0$
\[ N(x, y) = NV(x, y) + NW(x, y) + O\left(\frac{\Psi(x, y)^3}{x \cL^{c_5}}\right) \]
avec
\[ NV(x, y) := \sum_{q\leq R} \vphi(q) \int_{-1/(qQ)}^{1/(qQ)} V(x, y ; q, \eta)^2 V(x, y ; q, -\eta) \dd \eta \]
\[ NW(x, y) := \sum_{q\leq R} \DSZ(q) \vphi(q) \int_{-1/(qQ)}^{1/(qQ)} \left( 2V(x, y ; q, \eta) |W(x, y ; q, \eta)|^2 + V(x, y ; q, -\eta) W(x, y ; q, \eta)^2 \right) \dd \eta .\]
En appliquant les estimations~\eqref{majo_V_tv} et~\eqref{majo_W} et en remarquant que~$(1+|\eta|x)^{1-\alpha}=O(1)$ et~$q^{1-\alpha}=O(1)$,
on obtient après intégration par rapport à~$\eta$ et sommation sur~$q$, pour un certain~$c_6>0$,
\begin{equation}\label{majo_NW}
NW(x, y) \ll \frac{8^{\omega(q_1)} \Psi(x, y)^3}{q_1 x} \left( \frac{1}{H(u)^\delta x^{2-2\beta}} + \frac{1}{\cL^{c_6}} \right)
\ll \frac{\Psi(x, y) \log u}{x \log y}
.\end{equation}

On fixe un réel~$\ee>0$ tel que~$1/\cY_{2\ee} \ll (\log u)/\log y$. Grâce à la majoration~\eqref{majo_V_tv}, on obtient
\begin{align*}
NV(x, y) = &\ \sum_{q \leq \cY_\ee} \vphi(q) \int_{-\cY_{\ee}/(qx)}^{\cY_{\ee}/(qx)} V(x, y ; q, \eta)^2 V(x, y ; q, -\eta) \dd \eta \\
&\ + O\Bigg( \sum_{q > \cY_\ee} \frac{8^{\omega(q)} (\log q)^2 \Psi(x,y)^3}{\vphi(q)^2} \int_{0}^{1/(qQ)} \frac{(\log(2+|\eta|x))^3\dd\eta}{(1+|\eta|x)^3}\Bigg) \\
&\ + O\Bigg( \sum_{q \geq 1} \frac{8^{\omega(q)} (\log q)^2 \Psi(x,y)^3}{\vphi(q)^2} \int_{\cY_{\ee}/(qx)}^{\infty} \frac{(\log(2+|\eta|x))^3\dd\eta}{(1+|\eta|x)^3}\Bigg)
.\end{align*}
Les termes d'erreur sont~$\ll \Psi(x, y)^3/(x \cY_{2\ee})$.
La Proposition~\ref{caracts_pcp_saias} fournit alors, pour~$q \leq \cY_\ee$ et~$|\eta|x \leq \cY_{\ee} / (qx)$,
\[ V(x, y ; q, \eta)^2 V(x, y ; q, -\eta) = \Vt(x, y ; q, \eta)^2 \Vt(x, y ; q, -\eta)
+ O\left(\frac{8^{\omega(q)}\Psi(x, y)^3}{\vphi(q)^3(1+|\eta|x)^2\cY_{2\ee}}\right), \]
on obtient donc
\[ NV(x, y) = \sum_{q \leq \cY_\ee} \vphi(q) \int_{-\cY_{\ee}/(qx)}^{\cY_{\ee}/(qx)} \Vt(x, y ; q, \eta)^2 \Vt(x, y ; q, -\eta) \dd \eta
 + O\left(\frac{\Psi(x, y)^3}{x \cY_{2\ee}}\right) \]

On fait maintenant appel aux estimations suivantes, qui sont respectivement la formule~(4.22) et le lemme~5.1 de~\cite{AGRB2010}.
\begin{lemme}
Soit~$\ee>0$. Lorsque~$(x, y) \in (H_\ee)$, on a
\begin{equation}\label{AGRB_4.22}
\Vt(x, y ; q, \eta) \ll_\ee \frac{\cY_\ee^{2(1-\alpha)} u (\log u) 2^{\omega(q)}\Psi(x, y)}{\vphi(q) |\eta| x}
\qquad (q \leq \cY_\ee, \eta \neq 0),
\end{equation}
\begin{equation}\label{AGRB_lemme5.1}
\begin{aligned}
&\ \sum_{\substack{n\leq x \\ k|n}} \lambda\left(\frac{n}{k}, y\right) \int_{-1/(2k)}^{1/(2k)} \e((n'-n)\eta) \dd \eta \\
= &\ \frac{\lambda(n'/k, y)}{k}\bfUn_{[1, x]}(n') + O_\ee\left(\frac{ky/(\log y)}{\min\{|x-n'|+1, |n'-k|+1, n'\} }\right)
\quad (n'\in \bfN, k\leq \cY_\ee),
\end{aligned}
\end{equation}
\end{lemme}
\begin{proof}
D'après les calculs de la Proposition~\ref{caracts_pcp_saias} et en particulier l'estimation~\eqref{Vt_ipp}, on a pour tout diviseur~$k$ de~$q$,
\begin{align*}
\sum_{n\leq x/k} \e(nk\eta) \lambda(n, y) &\ =
E(x/k, x/k ; k\eta) \lambda(x/k, y) - \int_{x/(k\cY_\ee)}^x E(t, t ; k\eta) \lambda'(t, y) \dd t + O \left(\frac{\Psi(x, y)}{k \cY_{2\ee}}\right) \\
&\ \ll \frac{\cY_\ee^{2-2\alpha} u \log u \Psi(x, y)}{k|\eta|x}.
\end{align*}
En reportant cette majoration dans~\eqref{def_Vt}, on obtient l'estimation~\eqref{AGRB_4.22}.

En ce qui concerne l'estimation~\eqref{AGRB_lemme5.1}, on peut suivre la même preuve que celle de~\cite[lemme~5.1]{AGRB2010},
en remarquant que les seules estimations utilisées sont~$\lambda(t, y) \ll 1 \ \ (t, y \geq 2)$ ainsi que~\eqref{majo_lambda'},
toutes deux valables sous nos hypothèses.
\end{proof}
L'estimation~\eqref{AGRB_4.22} fournit
\begin{align*}
NV(x, y) =&\ \sum_{q\leq \cY_\ee} \vphi(q) \int_{-1/(2k_3)}^{1/(2k_3)} \Vt(x, y ; q, \eta)^2 \Vt(x, y ; q, -\eta)\dd \eta \\
&\ + O\left(\cY_\ee^{4-6\alpha} (u \log u)^3 \sum_{q\leq \cY_\ee} \frac{8^{\omega(q)} q^2}{\vphi(q)^2} \frac{\Psi(x, y)^3}{x}
+ \frac{\Psi(x, y)^3}{x \cY_{2\ee}}\right)
.\end{align*}
Sous nos hypothèses sur~$x$ et~$y$, on a~$u = \cY_{\ee}^{o(1)}$ et~$\alpha =1+o(1)$ lorsque~$x$ et~$y$ tendent vers l'infini, le terme d'erreur est donc~$O(\Psi(x, y)^3/(x \cY_{2\ee}))$.
En développant le terme~$\Vt(x, y ; q, -\eta)$, on écrit
\begin{equation}\label{NV_nv_R}
NV(x, y) = \sum_{q\leq \cY_\ee} \nv(x, y ; q) + R(x, y) + O\left(\frac{\Psi(x, y)^3}{x \cY_{2\ee}}\right)
\end{equation}
où l'on a posé, de même que dans~\cite{AGRB2010},
\[ \nv(x, y ; q) = \sum_{\substack{k_1, k_2, k_3 \in \bfN \\ k_i | q}}
\mu\left(\frac{q}{k_1}\right) \mu\left(\frac{q}{k_2}\right) \mu\left(\frac{q}{k_3}\right) \frac{k_1 k_2}{\vphi(q)^2}
\sum_{\substack{n_1, n_2 \in \bfN \\ n_1+n_2\leq x \\ k_i | n_i}}
\lambda\left(\frac{n_1}{k_1}, y\right) \lambda\left(\frac{n_2}{k_2}, y\right) \lambda\left(\frac{n_1+n_2}{k_3}, y\right) \]
et la majoration~\eqref{AGRB_lemme5.1} fournit
\begin{align*}
R(x, y) \ll &\ \frac{y}{\log y} \sum_{q \leq \cY_\ee} \frac{1}{\vphi(q)^2}
\sum_{k_1, k_2, k_3 | q} \mu^2\left(\frac{q}{k_1}\right)\mu^2\left(\frac{q}{k_2}\right) \mu^2\left(\frac{q}{k_3}\right) k_1 k_2 k_3^2 \\
&\ \sum_{\substack{n_1 \leq x \\ n_2 \leq x \\ k_i | n_i}} \lambda\left(\frac{n_1}{k_1}, y\right) \lambda\left(\frac{n_2}{k_2}, y\right)
\left\{\frac{1}{|x-n_1-n_2|+1} + \frac{1}{|n_1+n_2-k_3|+1} + \frac{1}{n_1+n_2}\right\}
.\end{align*}
En majorant trivialement~$\lambda(n_i/k_i, y)$ par~$O(1)$, puis la somme sur~$(n_1, n_2)$ par~$O(x \log x)$, on obtient
\begin{equation}\label{majo_R}
R(x, y) \ll x y u \cY_\ee^6 \ll \frac{\Psi(x, y)}{x \cY_{2\ee}}
.\end{equation}

Lorsque~$q y \leq n \leq x$, on a~$\lambda(n/k, y) \ll (kx/n)^{1-\alpha} \rho(u)$. Par ailleurs,
\[ \sum_{\substack{qy \leq n_i \leq x \\ k_i | n_i}} (n_1 n_2 (n_1+n_2))^{\alpha-1}
\leq \frac{1}{k_1 k_2} \int_0^x \int_0^{x} (t_1 t_2 (t_1+t_2))^{\alpha-1} \dd t_2 \dd t_1 \ll \frac{x^{3\alpha-1}}{k_1 k_2} .\]
En utilisant~$qy x \ll \Psi(x, y)^3/x$, on obtient~$\nv(x, y ; q) \ll 8^{\omega(q)} q^{3(1-\alpha)} \Psi(x, y)^3 / (\vphi(q)^2 x)$.
La série de terme général~$\nv(x, y ; q)$ est donc convergente et on a
\[ NV(x, y) = \sum_{q\geq 1} \nv(x, y ; q) + O\left(\frac{\Psi(x, y)^3}{x \cY_{2\ee}}\right) .\]
En utilisant la notation de~\cite[lemme 5.5]{AGRB2010}, on écrit
\begin{equation}\label{expr_sum_nv}
\sum_{q\geq 1} \nv(x, y ; q) = \sum_{(k_1, k_2, k_3) \in \bfN^3} g(k_1, k_2, k_3) S(k_1, k_2, k_3)
\end{equation}
où l'on a posé
\begin{equation}\label{def_S}
S(k_1, k_2, k_3) := \sum_{\substack{(n_1, n_2) \in \bfN^2 \\ k_i | n_i \\ n_1+n_2 \leq x}}
\lambda\left(\frac{n_1}{k_1}, y\right) \lambda\left(\frac{n_2}{k_2}, y\right) \lambda\left(\frac{n_1+n_2}{k_3}, y\right)
\end{equation}
et où~$g(k_1, k_2, k_3)$ vérifie
\[ \sum_{(k_1, k_2, k_3) \in \bfN^3} \frac{g(k_1, k_2, k_3)}{k_1 k_2} = 1, \qquad
\sum_{(k_1, k_2, k_3) \in \bfN^3} \frac{|g(k_1, k_2, k_3)|(k_1 k_2 k_3)^{1/4}}{k_1 k_2} \ll 1. \]
D'après ce qui précède, on a~$S(k_1, k_2, k_3) \ll (k_1 k_2 k_3)^{1-\alpha} \Psi(x, y)^3 / (k_1 k_2 x)$.
Il en découle que dans la somme du membre de droite de~\eqref{expr_sum_nv}, la contribution des triplets~$(k_1, k_2, k_3)$ avec~$k_1 k_2 k_3 \geq (\log y)^5$
est~$O(\Psi(x, y)^3/(x \log y))$.
Étant donné un triplet~$(k_1, k_2, k_3)$ vérifiant~$k_1 k_2 k_3 \leq (\log y)^5$, dans le membre de droite de~\eqref{def_S} :
\begin{itemize}
\item la contribution des~$(n_1, n_2)$ vérifiant~$n_1 \leq k_1 y$ ou~$n_2 \leq k_2 y$ est~$O((\log y)^5 y x)$,
ce qui est largement~$O(\Psi(x, y)^3 / (x \log y))$,
\item la contribution des~$(n_1, n_2)$ vérifiant~$k_2 y \leq n_2 \leq k_2 x / (\log y)^{6}$ et~$n_1 \geq k_1 y$ est
\begin{align*}
&\ \ll \frac{(k_1 k_2 k_3)^{1-\alpha}\Psi(x, y)^3}{x^{3\alpha}}
\sum_{\substack{k_1 y \leq n_1  \leq x \\ k_2 y \leq n_2 \leq k_2 x / (\log y)^{6} \\ k_i | n_i}} (n_1 n_2 (n_1+n_2))^{\alpha-1} \\
&\ \ll \frac{(k_1 k_2 k_3)^{1-\alpha} \Psi(x, y)^3}{k_1 k_2 x (\log y)^{2\alpha}} \ll \frac{\Psi(x, y)^3}{k_1 k_2 x \log y}
\end{align*}
puisque~$\alpha = 1+o(1)$ lorsque~$x$ et~$y$ tendent vers l'infini,
\item la contribution des~$(n_1, n_2)$ vérifiant~$k_1 y \leq n_1 \leq k_1 x / (\log y)^{6}$ et~$n_2 \geq k_2 y$ se majore de façon identique,
\item lorsque~$n / k \geq x/(\log y)^{6}$, on a
\begin{align*}
\lambda\left(\frac{n}{k}, y\right)  = \rho\left(\frac{\log(n/k)}{\log y}\right)\left\{ 1 + O\left(\frac{\log u}{\log y}\right) \right\} 
 = \rho(u) \left\{ 1 + O\left(\frac{(\log u) \log(kx/n)}{\log y} \right) \right\}
.\end{align*}
La contribution des~$(n_1, n_2)$ vérifiant~$n_i/k_i \geq x/(\log y)^{6}$ vaut donc
\[ \frac{\Psi(x, y)^3}{2x}\left\{ 1 + O\left(\frac{(\log u) \log(k_1 k_2 k_3)}{\log y}\right) \right\} \]
où l'on a utilisé la majoration~$\sum_{x/(\log y)^{6} \leq m \leq x/k} \log(x/m) \ll x(\log k)/k$.
\end{itemize}
En regroupant les résultats, on obtient
\[ \sum_{q\geq 1} \nv(x, y ; q) = \frac{\Psi(x, y)^3}{2x} \left\{ 1 + O\left(\frac{\log u}{\log y}\right) \right\} \]
et l'estimation~\eqref{estim_N} en découle en reportant cela avec~\eqref{majo_R} dans~\eqref{NV_nv_R}.

\end{proof}

\bibliographystyle{smfalpha}
\bibliography{expo_friable}

\end{document}